\documentclass[12pt]{amsart}
\usepackage{fullpage}

\input xy
\xyoption{all}

\newcommand{\ord}{\mbox{ord}}
\newcommand{\Gal}{\mbox{Gal}}

\newcommand{\id}{\mbox{id}}

\newcommand{\rank}{\mbox{rank}}

\usepackage{color, graphicx}
\usepackage{amsmath}
\usepackage{amssymb}
\usepackage{comment}

\usepackage[T1]{fontenc}
\usepackage{calrsfs}

\theoremstyle{plain}

\newtheorem{theorem}{Theorem}[section]
\newtheorem{lemma}[theorem]{\bf Lemma}
\newtheorem{corollary}[theorem]{\bf Corollary}
\newtheorem{proposition}[theorem]{\bf Proposition}

\theoremstyle{definition}
\newtheorem{definition}[theorem]{\bf Definition}
\newtheorem{question}[theorem]{\bf Question}

\newtheorem{example}[theorem]{\bf Example}

\theoremstyle{remark}
\newtheorem{remark}[theorem]{\bf Remark}

\newtheorem{notation}[theorem]{\bf Notation}
\newtheorem{notationassumptions}[theorem]{\bf Notation and Assumptions}

\newcommand{\calA}{{\mathcal A}}

\newcommand{\calC}{{\mathcal C}}

\newcommand{\calP}{{\mathcal P}}
\newcommand{\calQ}{{\mathcal Q}}

\newcommand{\calS}{{\mathcal S}}

\newcommand{\calW}{{\mathcal W}}

\newcommand{\calZ}{{\mathcal Z}}

\newcommand{\C}{{\mathbb C}}

\newcommand{\Q}{{\mathbb Q}}
\newcommand{\R}{{\mathbb R}}
\newcommand{\Z}{{\mathbb Z}}

\newcommand{\pp}{{\mathfrak p}}
\newcommand{\rr}{{\mathfrak r}}

\newcommand{\qq}{{\mathfrak q}}

\newcommand{\nn}{{\mathfrak n}}
\newcommand{\dd}{{\mathfrak d}}

\newcommand{\ttt}{{\mathfrak t}}
\newcommand{\aaa}{{\mathfrak a}}

\newcommand{\be}{\begin{enumerate}}
\newcommand{\ee}{\end{enumerate}}

\newcommand{\inff}{{\mbox{\small inf}}}

\def\ord{\mathop{\mathrm{ord}}\nolimits}
\def\inff{\mathop{\mathrm{inf}}\nolimits}

\usepackage[OT2,OT1]{fontenc}
\newcommand\cyr{%
\renewcommand\rmdefault{wncyr}%
\renewcommand\sfdefault{wncyss}%
\renewcommand\encodingdefault{OT2}%
\normalfont
\selectfont}
\DeclareTextFontCommand{\textcyr}{\cyr}

\begin{document}
\bibliographystyle{plain}%
 \title{First Order Decidability and Definability of Integers in Infinite Algebraic Extensions of the Rational Numbers}%
\author{Alexandra Shlapentokh}%
\thanks{The research for this paper has been partially supported by NSF grants DMS-0650927 and DMS-1161456,  and a grant from John Templeton Foundation.  The author would also like to thank Carlos Videla for helpful comments.}
\address{Department of Mathematics \\ East Carolina University \\ Greenville, NC 27858}%
\email{shlapentokha@ecu.edu }
\urladdr{www.personal.ecu.edu/shlapentokha} 

\date{\today}

\maketitle

\bigskip

\bigskip

\begin{abstract}
We extend results of Videla and Fukuzaki to define algebraic integers in large classes of infinite algebraic extensions of $\Q$ and use these definitions for some of the fields to show the first-order undecidability. We also obtain a structural sufficient condition for definability of the ring of integers over its field of fractions.  In particular, we show that the following propositions hold.  (1) For any rational prime $q$  and any positive rational integer $m$, algebraic integers are definable in any Galois extension of $\Q$ where the degree of any finite subextension is not divisible by $q^{m}$. (2)  Given a prime $q$, and an integer $m>0$,  algebraic integers are definable in a cyclotomic extension (and any of its subfields) generated by any set $\{\xi_{p^{\ell}}| \ell \in \Z_{>0}, p \not=q \mbox{ is any prime such that } q^{m +1}\not | (p-1)\}$.  (3) The first-order theory of any abelian extension of $\Q$ with finitely many  ramified rational primes is undecidable.  We also show that under a condition on the splitting of one rational prime in an infinite algebraic extension of $\Q$, the existence of a finitely generated elliptic curve over the field in question is enough to have a definition of $\Z$ and to show that the field is indecidable. 

\end{abstract}%

\section{Introduction}
The purpose of this paper is to consider the following problems of definability and decidability for an infinite algebraic extension $K_{\inff}$ of $\Q$. 
 \begin{question}
 \label{def}
 Is the ring of integers of $K_{\inff}$ first-order definable over $K_{\inff}$?
 \end{question}
 \begin{question}
 \label{decide}
 Is the first-order theory of $K_{\inff}$ decidable?
 \end{question} 
 The questions of this type have a long history, especially as applied to number fields and in connection to generalizations of Hilbert's Tenth Problem.  We will not attempt to give a full accounting of the work done in the subject here but will limit ourselves to pointing out some surveys as well as results specifically relevant to this paper.  

Perhaps a good place to start is with the results of J. Robinson who proved in \cite{Rob1} and \cite{Rob2} that in any number field the ring of integers of the number field as well as the ring of rational integers are first-order definable in the language of rings, and therefore the first-order theory of these fields (in the language of rings) is undecidable.  In the process of proving these results J. Robinson also proved that integrality at a prime of a number field is existentially definable in the language of rings over a number field.  In \cite{Rob3} J. Robinson produced a uniform definition of $\Z$ over rings of integers of number fields,   R. Rumely in \cite{Rum} improved J. Robinson's result making the definition of the ring of integers over number fields uniform across fields.  More recently, B. Poonen in \cite{Po4} and J. Koenigsmann in \cite{Koenig2}
updated J. Robinson's definition of integers by reducing the number of universal quantifiers used in these definitions, B.  Poonen to two and J. Koenigsmann to one.  

The desire to reduce the number of universal quantifiers is motivated to large extent by the interest in extending Hilbert's Tenth Problem to $\Q$.  This would be accomplished by a purely existential definition of $\Z$ over $\Q$.  Unfortunately there are serious doubts as to whether such a definition exists.  See \cite{Dencollection},  \cite{Po7} and  \cite{Sh35} for surveys on Hilbert's Tenth Problem and related questions of definability.

A lot of work aiming to prove the decidability of the first-order theory has centered around various infinite extensions of $\Q$. (See \cite{Dencollection} for a survey of these results.) One of the more influential results was arguably due to R. Rumely in \cite{Rum1}, where he showed that Hilbert's Tenth Problem is decidable over the ring of all algebraic integers.  This result was strengthened by L. van den Dries proving in \cite{Dries4} that the first-order theory of this ring was decidable.  Another remarkable result is due to M. Fried, D. Haran and H. V{\"o}lklein in \cite{Fried3}, where it is shown that the first-order theory of the field of all totally real algebraic numbers is decidable.  This field constitutes a boundary of sorts between the ``decidable'' and ``undecidable'', since J. Robinson showed in \cite{Rob3}  that the first-order theory of the ring of all totally real integers is undecidable. In the same paper, she also proved that the first-order theory of a family of totally real rings of integers is undecidable and produced a ``blueprint'' for such proofs over rings of integers which are not necessarily totally real.  

Using some ideas of J. Robinson, an elaboration of J. Robinson's ``blueprint'' by C. W. Henson (see page 199 of \cite{Dries4}), and R. Rumely's method for defining integrality at a prime, C. Videla produced the first-order undecidability results for a family of infinite algebraic extensions of $\Q$ in \cite{V1}, \cite{V2} and \cite{V3}.  More specifically, C. Videla showed that the first-order theory of some totally real infinite quadratic extensions, any infinite cyclotomic extension with a single ramified prime, and some infinite cyclotomic extensions with finitely many ramified primes is undecidable.  C. Videla also produced the first result concerning definability of the ring of integers over an infinite algebraic extension of $\Q$ by generalizing a technique of R. Rumely: he showed that if all finite subextensions are of degree equal to a product of powers of a fixed (for the field) finite set of primes, then the ring of integers is first-order definable over the field.

In a recent paper \cite{Fuk}, K. Fukuzaki, generalizing further R. Rumely's method, proved that a ring of integers is definable over an infinite Galois extension of the rationals such that every finite subextension has odd degree over the rationals and its prime ideals dividing 2 are unramified.  He then used one of the results of J. Robinson to show that a large family of totally real fields contained in cyclotomics (with infinitely many ramified primes) has an undecidable first-order theory.

\section{The statements of new results and overview of the proofs}
The results of this paper can be divided into two categories: definability results, more specifically defining rings of integers and $\Z$ over infinite extensions,  and undecidability results for infinite extensions.  We discuss our new definability results first.
\subsection{The new definability results: $q$-boundedness} 
For the purposes of our discussion we fix an algebraic closure $\tilde \Q$ of $\Q$ and consider a progression from $\Q$ to its algebraic closure, first  through the finite extensions of $\Q$, next  through its infinite extensions fairly ``far''  from the algebraic closure, and finally through the infinite extensions of $\Q$ fairly ``close'' to $\tilde \Q$. 

 As one gets closer to $\tilde \Q$, there is an expectation that the language of rings would loose more and more of its expressive power.  It would be interesting to describe the mile posts signifying various stages of this loss.
A definitive description of these mile posts is probably far away, but in this paper we consider a candidate for an early mile post for the loss of definability, the loss of what we called ``$q$-boundedness'' for all rational primes $q$.  

The formal description of $q$-boundedness is in  Definition \ref{unbounded}.  We offer an informal sketch here.  Given an infinite algebraic extension $K_{\inff}$ of $\Q$ we consider what happens to the local degrees of primes over $\Q$ as we move through the factor tree within $K_{\inff}$.  A rational prime $p$  is called {\it $q$-bounded} if it lies on a path through the factor tree in $K_{\inff}$ where the local degrees of its factors over $\Q$ are {\bf not} divisible by arbitrarily high powers of $q$.  If every descendant of $p$ in every number field contained in $K_{\inff}$ has the same property, then we say that $p$ is hereditarily $q$-bounded.  

For  $q$ itself we require a stronger condition:  the local degrees along all the paths of the factor tree should have uniformly bounded order at $q$.  If this condition is satisfied, we say that $q$ (or some other prime in question) is completely {\it $q$-bounded}.  If all the primes $p \not =q$ are hereditarily $q$-bounded and $q$ is competely $q$-bounded, we say that the field $K_{\inff}$ itself is $q$-bounded, and  we show in Theorem \ref{thm:maindef} that the ring of integers is definable in such a field.  Rings of integers are also definable under some modifications of the $q$-boundedness assumptions, such as an assumption that all primes $p \not =q$ are hereditarily $q$-bounded and $q$ is completely $t$-bounded for some prime $t \not=q$, etc.  We give a sample of results of this type in Theorems \ref{thm:maindef} and \ref{general}.  We also show that one can leverage the $q$-unbounded primes for the purposes of definability, i.e. to define rings where only $q$-unbounded primes can appear in the denominator. (See Theorem \ref{unbounded}.)  

Below we explain what new fields our results cover, but perhaps the most important aspect of our definability result is the structural one.  We suspect that $q$-boundedness or a similar condition, e.g. a somewhat more general condition described in Theorem \ref{general}, is  necessary for definability of the ring of integers.  While non-definability examples are scarce over infinite extensions, we offer the following ones: the field of all totally real numbers, not satisfying the assumptions of Theorem \ref{general}, and decidable by the result of M. Fried, H. V\"{o}lkline, and D. Haran (\cite{Fried3}), has the ring of integers  not definable over the field, since it is undecidable by a result of J. Robinson (\cite{Rob2}).  Further, the field of real algebraic number also does not satisfy the assumptions of Theorem \ref{general}, and its ring of integers is not definable over the field by a result of A. Tarski (\cite{Tar}).
\subsection{Results of C. Videla and K. Fukuzaki considered within the $q$-bounded framework}
Proceeding chronologically, we reconsider results of C. Videla first.   As mentioned above, his results concerned infinite Galois extensions  of number fields, where all the finite subextensions are of degree divisible only by primes belonging to a fixed {\it finite} set of primes $A$.  Consequently, in the fields considered by C. Videla all the primes are completely $q$-bounded for  any $q \not \in A$, and thus all these fields are certainly $q$-bounded. 

 The first natural extension of C. Videla's result, obtainable from our work, is the proposition that the integers are definable in any Galois extension where all the finite subextensions  have degree  {\it not} divisible by a {\it single} prime $q$ (while results of C. Videla prohibit divisibility of the degrees by all but finitely many primes).    Further we can allow finitely many subextensions to be divisible by $q$.  For this reason, while C. Videla could show that the ring of algebraic integers is definable in any cyclotomic extension of $\Q$ with finitely many ramified primes, we can show that the ring of integers is definable in a larger class of cyclotomic extensions, including extensions with infinitely many ramified primes.  For example, for any rational prime $q$ and any  $m \in \Z_{>0}$ we can adjoin to $\Q$ all $\ell^n$-th roots of unity for any positive integer $n$ and for any rational prime $\ell$ such that $q^m$ does not divide $\ell-1$.

	Turning our attention to K. Fukuzaki we note that all the fields he considers are 2-bounded.   Further, K. Fukuzaki does not allow any ramification of dyadic ideals and no finite subextensions  of even degrees, options we can allow even if we just consider 2-bounded fields.  Thus, again as described above, K. Fukuzaki's results allow him to consider some totally real subfields of cyclotomics with infinitely many ramified primes but not the  cyclotomics themselves.

Further, both C. Videla and K. Fukuzaki consider only Galois extensions, a restriction we do not require.  Many more examples of $q$-bounded fields, some natural and some less so can be found in Section \ref{examples}.  Among a set of natural examples not covered by earlier work are non-Galois fields that are towers of finite subextensions of degrees less than $m$ for some positive integer $m$.  We should also note that the family of fields we consider is closed under any finite extension, a property not shared by the fields considered by earlier researchers in the area.
 Finally all our definability results are proved more generally for the rings of $\calS$-integers for an arbitrary finite $\calS$, with empty $\calS$ corresponding the ring of integers.
\subsection{Overview of the construction of our definition of integers}
The central part of our construction is a norm equation which has no solutions if a field element in question has ``forbidden'' poles.  (In an effort to simplify terminology we transferred some function field terms   to this number field setting.)   While we are far from being the only or the first practitioners of this method which originates with J. Robinson and R. Rumely, we do employ  a unique,  to our knowledge,  variation of it.  More specifically, as explained below, we do not fix the top or the bottom field in the norm equation, but allow these fields to vary depending on the elements involved.  As long as the degree of all extensions involved is bounded, such a ``floating'' norm equation is still (effectively) translatable into a system of polynomial equations over the given field.  (See the proof of Theorem \ref{thm:maindef} for the description of this translation).   

To set up the norm equation, let 
\begin{itemize}
\item $q$ be a rational prime number,
\item $K$ be a number field containing a primitive $q$-th root of unity,
\item $\pp_K$ be a prime of $K$ not dividing $q$, 
\item  $b \in K$ be such that $\ord_{\pp_K}b=-1$,
\item $c \in K$ be such that $c$ is integral at $\pp_K$ and is not a $q$-th power in the residue field of $\pp_K$,
\end{itemize}
 and consider $bx^q+b^q$.  Note that $\ord_{\pp_K}(bx^q +b^q)$ is divisible by $q$ if and only if $\ord_{\pp_K}x \geq 0$.  Further, if $x$ is an integer,  all the poles of $bx^q+b^q$ must be poles of $b$ and are divisible by $q$.    Assume also that all  zeros of $bx^q+b^q$ and all zeros and poles of $c$ are of orders divisible by $q$ and $c \equiv 1 \mod q^3$.  Finally, to simplify the situation further, assume that either $K$ has no real embeddings into $\tilde \Q$ or $q >2$.  Now consider the norm equation
\begin{equation}
\label{norm:prelim}
{\mathbf N}_{K(\sqrt[q]{c})/K}(y)=bx^q+b^q.
\end{equation}
Since $\pp_K$ does not split in this extension, if $x$ has a pole at $\pp_K$, then $\ord_{\pp_K}bx^q+b^q \not \equiv 0 \mod q$, and the norm equation has no solution $y$ in $K(\sqrt[q]{c})$.  Further, if $x$ is an integer, given our assumptions, using the Hasse Norm Principle we can show that this norm equation does have a solution. Our conditions on $c$ insure that the extension is unramified, and  our conditions on $bx^q+b^q$ in the case $x$ is an integer make sure that locally {\it at every prime not splitting in the extension} the element $bx^q+b^q$ is equal to  a $q$-th power of some element of the local field times a unit.  By the Local Class Field Theory, this makes  $bx^q+b^q$ a norm locally at every prime.

For an arbitrary $b$ and $c\equiv 1 \mod q^3$ in $K$, we will not necessarily have  all  zeros of $bx^q+b^q$ and all zeros and poles of $c$ of orders divisible by $q$.   For this reason, given $x, b, c \in K$ we consider our norm equation in a finite extension $L$ of $K$ and this extension $L$ depends on $x, b, c$ and $q$.  We choose $L$ so that all  primes occurring as zeros of $bx^q+b^q$ or as  zeros or poles of $c$ are ramified with ramification degree divisible by $q$.   We also take care to split  $\pp_K$ completely in $L$, so that in $L$ we still have that $c$  is not a $q$-th power modulo any factor of $\pp_L$.  This way, as we run through all $b, c \in K$ with $c-1 \equiv  0 \mod q^3$, we ``catch'' all the primes  that do not divide $q$ and occur as poles of $x$.  The construction of the field $L$ and the argument concerning the properties of the primes in question in this field are in Propositions \ref{prop:badprime} and \ref{prop:fixorder}. 

Unfortunately, we will not catch factors of $q$ that may occur as poles in this manner, because our assumption on $c$ forces all the factors of $q$ to split into distinct factors in the extension.   Splitting factors of $q$ into distinct factors protects us from a situation where such primes may ramify and cause the norm equation not to have solutions even when $x$ is an integer.  Elimination of factors of $q$ from the denominators of the divisors of the elements of the rings we define will be done separately. 

The end result of this construction, described in detail in  Section \ref{Algebraic}, is essentially a uniform definition of the form $\forall \forall \exists \ldots \exists$  of the ring of $\calQ$- integers, with $\calQ$ containing factors of $q$, across all number fields containing the $q$-th primitive roots of unity.   

Putting aside for the moment the issue of defining the set of all elements $c$ integral at $q$ and equivalent to 1 mod $q^3$, and the related issue of defining integrality at factors of $q$ in general, we now make the transition to an infinite $q$-bounded extension $K_{\inff}$ by noting the following.  Let $K \subset K_{\inff}$, let $\pp_K$ be a prime of $K$ such that $\pp_K$ does not divide $q$, let $x \in K$ and let $\ord_{\pp_K}x <0$.  Since by assumption $\pp_K$ is $q$-bounded, it lies along a path in its factor tree within $K_{\inff}$, where the order at $q$ of local degrees eventually stabilizes.  To simplify the situation once again, we can assume that it stabilizes immediately past $K$.  So let $N$ be another  number field with $K \subset N \subset K_{\inff}$.  In this case for some prime $\pp_N$ above $\pp_K$ in $N$, we have that $\ord_q e(\pp_N/\pp_K)=\ord_qf(\pp_N/\pp_K)=0$.  Now,  let $b,c \in K$ be as above and observe that $c$ is not a $q$-th power in the residue field of $\pp_N$ while $\ord_{\pp_N}(bx^q+b^q) \not \equiv 0 \mod q$.  Thus the corresponding norm equation with $K$ replaced by $N$  and eventually by $K_{\inff}$ in \eqref{norm:prelim} has no solution.  Of course when $x$ is an integer and we have a solution to our norm equation in $K$, we also have a solution in $K_{\inff}$.

Note that for each prime $\pp_K$ of $K$, at every higher level of the tree we need just one factor with the local degree  not divisible by $q$ to make the norm equation unsolvable when $\pp_K$ appears in the denominator of the divisor of $x$.  Hence having one $q$-bounded path per every prime of $K$ is enough to make sure that no prime  of $K$ not dividing $q$ occurs as a pole of any element of $K$ in our set.   

Unfortunately, if we go to an extension of $K$ inside $K_{\inff}$, some primes of $K$ will split into distinct factors and can occur independently in the denominators of the divisors of elements of extensions of $K$.  Thus, in the extensions of $K$ inside $K_{\inff}$ we have to block each factor separately.  This is where the ``hereditary'' part comes in.  We need to require the same condition of $q$-boundedness for every descendant in the factor tree of  every prime of $K$ not dividing $q$, insuring integrality at all factors of all $K$-primes not dividing $q$.  

Before we tackle integrality at factors of $q$, we point out that a preliminary definition of the subring of an infinite extension containing only algebraic numbers with no poles outside the set of factors of $q$ is in \eqref{eq:C}.  Note that $\Phi_q(K_{\inff})$ is precisely the set of all $c \in K_{\inff}$ integral at $q$ and equivalent to 1 mod $q^3$.  Once we have a definition of  integrality at factors of $q$, we will also be able to define $\Phi_q(K_{\inff})$.

The main reason that only one $q$-bounded path per prime not dividing $q$ is enough to construct a definition of integers, is that the failure of the norm equation to have a solution locally at any one prime is enough  for the equation not to have solutions globally.  Conversely, in order to have solutions globally, we need to be able to solve the norm equations locally at all primes.  As already mentioned above, the reason we require $c$ to be integral at $q$ and equivalent to 1 mod $q^3$ is to make sure that factors of $q$ do not ramify when we take the $q$-th root of $c$. Just making $c$ have order divisible by $q$ at all primes does not in general guarantee that factors of $q$ do not ramify in such an extension.   If any factor of $q$ does ramify, then not all local units at this factor are norms in the extension, and making sure that the right side of the norm equation has order divisible by $q$ at all primes might not be enough to guarantee a global solution.   Hence we need to control the order of $c-1$ at {\it all} factors of $q$ at every level of the factor tree simultaneously, necessitating a stronger assumption on $q$, than on other primes. 

Depending on the field we might have a couple of options as far as integrality at $q$ goes.  If $q$ happens to be completely $p$-bounded in our infinite extension for some $p \not =q$, then we can pretty much use the same method as above with $p$-th root replacing the $q$-th root.  The only difference is that, assuming we have the primitive $p$-th root of unity in the field, by definition of a complete $p$-boundedness, we can fix  an element $c$ of the field such that $c$ is not a $p$-th power modulo any factor of $q$ in any finite subextension of $K_{\inff}$ containing some fixed number field.  We can also fix an element $b$ of the field such that the order of $b$ at any factor of $q$ is not divisible by $p$  in any finite subextension of $K_{\inff}$ containing the same fixed number field as above.  Using such elements $c$ and $b$ we can get an {\it existential} definition of a subset of the field containing all elements  with the order at any factor of $q$ bounded from below by a bound depending on $b$ and $p$.  (See Proposition \ref{finitely many for p not q}.)    If  ramification degrees of factors of $q$ are altogether bounded, then we can arrange for this set to be the set of all field elements integral at factors of $q$, but in a general case  the bound from below will be negative.  In this case to obtain the definition of integrality we will need one more step as described in Lemma \ref{integral}.

Before going back to infinite extensions, we would like to make a brief remark about the sets definable by our methods over number fields.  First of all, over any number field all primes are completely $p$-bounded for every $p$, and the ramification degree of factors of $q$ is altogether bounded.  So we can produce an existential and uniform (with parameters) definition  of integrality at all factors of $q$.  Note also that the complement of such a set is also uniformly existentially definable with parameters using the same method.  So, in summary, we now obtain a uniform definition of the form $\forall \forall \exists \ldots \exists$  of the ring of integers of any number field with a $q$-th primitive root of unity.  This result is along the lines of B. Poonen's result in \cite{Po4}, though his method is slightly different from ours since it uses ramified primes rather than non-splitting primes to obtain integrality formulas and restricts the discussion to $q=2$ and quadratic forms. As B. Poonen, we can also use $q=2$ and thus have a two-universal quantifier formula uniformly covering all number fields, but in this case if $K$ has real embeddings, we need to make sure that $c$ satisfies some additional conditions in order for the norm equations to have solutions (below we refer to these conditions as ``making sure that $c \in \Omega_2(K)$'').  

Returning now to the case of infinite extensions, we note that, assuming $q$ is $p$-bounded we now have a uniform first-order definition with parameters of algebraic integers across all $q$-bounded algebraic extensions of $\Q$ where $q$ is completely $p$-bounded.  However, for the infinite case we may require more universal quantifiers.  The number of these universal quantifiers will depend on the whether the ramification degree of factors of $q$ is bounded and on whether $q$ has a finite number of factors.  

The only case left to consider now is the case where $q$ is not completely $p$-bounded for any $p \not=q$ but is completely $q$-bounded.  This case requires a somewhat more technically complicated definition than the case where we had a requisite $p$.  In particular, we still need a cyclic extension (once again of degree $q$), where all the factors of $q$ will not split.  Such an extension does exist, but we might have to extend our field to be in a position to take advantage of it.  This construction is executed in  Lemma \ref{onlyq}.

\subsection{Overview of our construction defining $\Z$ using finitely generated elliptic curves and one completely $q$-bounded prime.}
 This section has an overview of a  construction of a definition of a number field $K$ over an infinite algebraic extension $K_{\inff}$ of $\Q$ using an elliptic curve with a Mordell-Weil group generated by points defined over $K$.  This construction also requires one completely $q$-bounded prime $p$ (which may equal to $q$).  Once we have a definition of $K$, a definition of $\Z$ follows from a result of J. Robinson. The use of elliptic curves for the purposes of definability also has a long history, as long as the one for norm equations and quadratic forms.  We review some of this history at the beginning of Section \ref{elliptic curves}.  Here we briefly dwell on the construction itself.

The main idea of the construction can be described as follows.  Given an element $x \in K_{\inff}$, we write down a statement saying that $x$ is integral at  $p$ and for every $n \in \Z_{>0}$ we have that $x$ equivalent to some element of $K$ mod $p^n$.  By the weak vertical method, this is enough to ``push'' $x$ into $K$.  (See Proposition \ref{extensions}.)   Our elliptic curve is the source of elements of $K$.  Any solution to an affine equation $y^2=x^3+ax+b$ of our elliptic curve must by assumption be in $K$.  Further if we let $P$ be a point of infinite order and let the affine coordinates of $[n]P$ corresponding to our equation be $(x_n, y_n)$, then the following statements are true: 
\be
\item Let $\mathfrak A$ be any integral divisor of $K$ and let $m$ be a positive integer.  Then there exists  $k \in \Z_{>0}$ such that $\mathfrak A \Big{|} \dd(x_{km})$, where $\dd(x_{km})$ is the denominator of the divisor of $x_{km}$ in the integral divisor semigroup of $K$. (See Lemma \ref{le:anydivisor}.)
\item There exists a positive integer $m$ such that for any positive integers $k,l $,
\[%
\dd(x_{lm}) \Big{|} \nn\left (\frac{x_{lm}}{x_{klm}}-k^2\right)^2
\]%
in the integral divisor semigroup of $K$. Here $\dd(x_{lm})$ as above refers to the denominator of the divisor of $x_{lm}$ and $\nn\left (\frac{x_{lm}}{x_{klm}}-k^2\right)$ refers to the  numerator of the divisor of  $\frac{x_{lm}}{x_{klm}}-k^2$.  (See Lemma \ref{le:equiv}.)  
\ee
Given $u \in K_{\inff}$  integral at some fixed $K$-prime $\pp_K$, we now consider a statement of the following sort: $\forall z \in K_{\inff}$ there exists $x, y, \hat x, \hat y \in K_{\inff}$ s.t. $(x,y), (\hat x, \hat y)$ satisfy the chosen elliptic curve equation and both $\displaystyle \frac{1}{zx}$ and  $\displaystyle x(u^2 -\frac{x}{\hat x})^2$ are integral at $\pp_K$ implying  that $\frac{(u^2 -\frac{x}{\hat x})^2}{z}$ is integral at $\pp_K$.  

If $u$ satisfies this formula, then since $\frac{x}{\hat x} \in K$, by the weak vertical method we have that $u \in K$.  Further, if $u$ is a square of an integer, this formula can be satisfied.  Thus we can proceed to define all integers, followed by all rational numbers and eventually $K$.  Finally, being able to define $\Z$ implies undecidability of the first-order theory of the field.

\subsection{Overview of the proof of undecidability of fields via undecidability of the rings of Integers and $\calS$-integers}

		As K. Fukuzaki we obtain first-order undecidability results using results of J. Robinson for totally real fields.  However we are also  able to use existential  undefinability results previously obtained by the author to show that the first order theory of fields and rings of integers of any abelian extension with finitely many ramified primes is undecidable, thus extending results of C. Videla.  The undecidability results are in Theorems \ref{undecidable} and \ref{cyclotomics} and Corollary \ref{cor:int}.  

To be more specific, a result of J. Robinson implies that if a ring of integers has a certain invariant which C. Videla called a ``Julia Robinson number'', one can define a first-order model of $\Z$ over the ring.  The Julia Robinson number $s$ of a ring $R$ of totally real integers is a real number  $s$
or $\infty$, such that $(0,s)$ is the smallest interval containing infinitely many sets of conjugates of numbers of
$R$, i.e., infinitely many $x \in R$ with all the conjugates (over $\Q$) in $(0, s)$.  A result of Kronecker implies that $s \geq 4$, and therefore if a totally real ring of integers in question contains the real parts of infinitely many distinct roots of unity, the Julia Robinson number for the ring is indeed 4, and we have the desired undecidability result.  Using our definability results, we thus obtain large families of totally real $q$-bounded fields with undecidable first-order theory.  Further, we can also show that every $q$-bounded totally real field is contained in an undecidable totally real field, while as we pointed out before, the field of all totally real numbers is decidable.

To use the existential undecidability results for rings, we need to define the integral closures of the rings of $\calS$-integers of number fields in infinite extensions under consideration.  This construction is necessary because the existential undecidability results previously obtained by the author pertain only to these bigger rings and not to the rings of integers.    The definitions of bigger rings require a minor adjustment of our construction above: we have to make $c$ as above equivalent to 1 no just modulo $q^3$ but also modulo all the primes in $\calS$.  Further, as in the case of $q$ and for similar reasons, we need primes in $\calS$ to be completely $q$-bounded.

\subsection{The structure of the paper}

The paper is structured in the following manner.  In Section \ref{Algebraic} we describe most of the algebraic number theory necessary to establish our results.   In Section \ref{primes} we discuss the conditions on primes we need to carry out our proofs over infinite extensions. This is where we introduce the notion of a field being ``$q$-bounded''. Section \ref{inf ext} completes the construction of first-order definitions of the rings of algebraic integers in specified infinite algebraic extensions of $\Q$, and Section \ref{examples} contains various examples of fields satisfying the requirements for our definitions.  Section \ref{undecidability} uses definitions of integers to produce undecidability results for fields.  Finally, Section \ref{elliptic curves} explains how to use finitely generated elliptic curves to obtain definitions of rational integers.

\section{Some Algebraic  Number Theory}
\label{Algebraic}
 In this section we show how to define a set of  elements of a number field containing all integers and such that all non-integers in the set have negative orders (poles) of order divisible by a given prime number $q$ only.  We start with some notation.
\begin{notationassumptions}
\label{not:1}
The following notation are used throughout the rest of the paper.
\begin{itemize}
\item Let $q$ be a rational prime number.
\item Let $ \xi_q$ be a primitive $q$-th root of unity.
\item Let $K, F, G, L$ denote algebraic extensions of $\Q$.
\item For a number field $G$, let $\pp_G, \qq_G, \ttt_G, \aaa_G$ be distinct non-archimedean primes of $G$.
\item If $K$ is any finite extension of a number field $G$, then $\pp_K, \qq_K, \ttt_K, \aaa_K$  denote primes above $\pp_G, \qq_G, \ttt_G, \aaa_G$ respectively.
\item For $K$ and $G$ as above, let $\calC_{K}(\pp_G)$ denote the set of all $K$-primes above $\pp_G$.
\item If $K$ is a number field and $x \in K$ and $\ord_{\pp_K}x >0$, we  say by analogy with function fields that $x$ has a zero at $\pp_K$.  Similarly, if $\ord_{\pp_K}x <0$, we say that $x$ has a pole at $\pp_K$.
\item If $\calS_K$ is a set of non-archimedean primes of $K$, then we let $O_{K,\calS_K}$ denote a subring of $K$ containing all the elements of $K$ without any poles at primes outside $\calS_K$.

\item For $x,b, d, a,c \in K\setminus \{0\}$, such that $bx^q+b^q\not=0, dx^q+d^q \not=0$   let
\[
L_1=K(\sqrt[q]{1+ x^{-1}}),
\]
\[
L_2=L_1(\sqrt[q]{1+ (bx^q+b^q)^{-1}}),
\]
\[
L=L_2(\sqrt[q]{1+(c+c^{-1})x^{-1}}),
\]
\[
F_1=K(\sqrt[q]{1+ d^{-1}}),
\] 
\[
F_2=F_1(\sqrt[q]{1+ (dx^q+d^q)^{-1}}),
\] 
\[
F=F_2(\sqrt[q]{1+(a+a^{-1})d^{-1}}).
\] 
and observe that $L$ depends on $K, q, x, b, c$, while $F$ depends on $K, q, a, x, d$.  For the rest of this section we will assume that $x,b, d, a,c$ take values in $K$ so that all the fields above are defined.

\item Let $\tilde \Q$ be an algebraic closure of $\Q$.
\item If $K$ is a number field, then for any prime $\pp_K$, let  $K_{\pp_K}$ be the completion of $K$ under the $\pp_K$-adic topology.
\item  If $K$ is a number field, and 
 \[
 \calS_K=\{\pp_{1,K}.\ldots,\pp_{l, K}\}
 \]
  is a finite set of primes of $K$, then let $\Theta_q(K,\calS_K)$ denote the set of all elements $c$ of $K$ such that the numerator of the divisor of $c-1$ is divisible by  the divisor $\prod_{i=1}^l \pp_{i,K}$ in the semi-group of the integral divisors of $K$.  If $\calS_K =\emptyset$, then set $\Theta_q(K,\calS_K)=K$.  
\item If $K$ is a number field, let $\Phi_q(K)$ denote the set of all elements $c$ of $K$ such that the numerator of the divisor of $c-1$ is divisible by $q^3$.  
 
\item  If $K$ is an infinite extension of $\Q$, and $\calS_K$ is a set of valuations of $K$ lying above finitely many primes of $\Q$, then  a $K$-element $c$ is  in $\Theta_q(K,\calS_K)$ if and only if for some number field $M \subset K$ and  the set $\calS_M$ of primes of $M$ below valuations of $\calS_K$ we have that $c \in \Theta_q(M,\calS_M)$.  Similarly a $K$-element $c \in \Phi_q(K)$ if and only if $c \in \Phi_q(\Q(c))$.
\item For an algebraic extension $K$ of $\Q$, let $\Omega_2(K)$ be the set of all the elements $c$ of $K$ such that for any embedding $\sigma$ of $K$ into $\tilde \Q$ we have that $\sigma(K) \subset \R\cap \tilde\Q$ implies $\sigma(x) \geq 0$.  If $K$ has no real embeddings or $q >2$, let $\Omega_q(K)=K$.
\end{itemize}
\end{notationassumptions}

The proof of the lemma below follows from the Hasse-Minkowski Theorem and the fact that over a local field a quaternary form is universal.
\begin{lemma}
\label{always non-negative}
If $H$ is any algebraic extension of $\Q$, then the set 
\[
\{x \in H|\exists x_1,x_2, x_3, x_4 \in H: x=x_1^2+x_2^2 + x_3^2 +x_4^2\}
\]
 is exactly the set of all elements of $H$ such that for any embedding $\sigma$ of $H$ into $\tilde \Q$ with $\sigma(H) \subset \R \cap \tilde \Q$ we have that $\sigma(x) \geq 0$.
\end{lemma}

\begin{remark}
\label{rem:omega}
If $K/M$ is an algebraic extension and $c \in \Omega_2(M)$, then $c \in \Omega_2(K)$. However, $\Omega_2(K) \cap M \not = \Omega_2(M)$ in all cases,  since there can be an embedding of $K$ into $\tilde \Q$ which is not real but the  restriction to the image of $M$ is real.  At the same time, if $K_{\inff}$ is an infinite algebraic extension of $M$ and $c \in \Omega_{2}(K_{\inff}) \cap M$, then for some finite extension $N$ of $M$ with $N\subset K_{\inff},$ for all $K$ such that $N \subseteq K \subset K_{\inff},$ we have $c \in \Omega_2(K)$.
\end{remark}

Next we state Hensel's lemma and its corollary which play an important role in  our use of Hasse Norm Principle.
\begin{lemma}
If $K$ is a number field, $f(X) \in K_{\pp_K}[X]$ has coefficients integral at $\pp_K$ and for some $\alpha \in K_{\pp_K}$ integral at $\pp_K$ we have that $\ord_{\pp_K}f(\alpha) > 2\ord_{\pp_K}f'(\alpha)$, then $f(X)$ has a root in $K_{\pp_K}$.  (See \cite{L}[Proposition 2, Section 2, Chapter II].)
\end{lemma}
\begin{corollary}
\label{cor:q-thpower}
 If  $K$ is a number field,  $x \in K$ is integral at all factors of $q$, $x \equiv 1 \mod q^3$, and $\qq_K$ is any prime of $K$ dividing $q$, then $x$ is a $q$-th power in $K_{\qq_K}$.
\end{corollary}
\begin{proof}
Let $f(X)= X^q-x$ and observe that by our assumption on $x$ we have the following:
\[
\ord_{\qq_K}f(1)=\ord_{\qq_K}(1-x) =3e(\qq_K/q).  
\]
At the same time $\ord_{\qq_K}f'(1)=\ord_{\qq_K}q=e(\qq_K/q)$ and therefore $\ord_{\qq_K}f(1) > 2\ord_{\qq_K}f'(1)$.  Hence, by Hensel's lemma $f(x)$ has a root in $K_{\qq_K}$, making $x$ a $q$-th power.
\end{proof}
The  two lemmas below, stated without a proof, list some basic number-theoretic facts.

\begin{lemma}
\label{le:first}
If $F$ is a number field containing $\xi_q$, $b \in F$ and $b$ is not a $q$-th power in $F$, then the following statements are true.
\begin{enumerate}
\item \label{firstit:1} If $\ord_{\pp_F}q=\ord_{\pp_F}b=0$, then $\pp_F$ does not ramify in the extension $F(\sqrt[q]{b})/F$.
\item \label{firstit:2} If $\ord_{\pp_F}b=0$, $b$ is not a $q$-th power  mod $\pp_F$,  and $\pp_F$ does not divide $q$, then $\pp_F$ does not split (i.e. has only one prime above it) in the extension $F(\sqrt[q]{b})/F$.
\item \label{firstit:3} If $\ord_{\pp_F} b =0$, $\pp_F$ does not divide $q$,   and $b$ is a $q$-th power  mod $\pp_F$, then $\pp_F$ splits into distinct factors in the extension $F(\sqrt[q]{b})/F$.
\item \label{firstit:4} If $\ord_{\pp_F}b \not \equiv 0 \mod q$, then $\pp_F$ ramifies completely in the extension $F(\sqrt[q]{b})/F$.
\end{enumerate}
\end{lemma}
The second lemma deals with norms and primes in cyclic extensions of degree $q$.
\begin{lemma}
\label{le:cycextensions}
Let $G/F$ be a cyclic extension of degree $q$ of number fields.  If $\pp_F$ is not ramified in the extension, then either it splits completely (in other words into $q$ distinct factors) or it does not split at all.  Further if $w ={\mathbf N}_{G/F}(z)$ for some $z \in G$, and $\pp_F$ does not split in the extension,  then $\ord_{\pp_F}w \equiv 0 \mod q$.
\end{lemma}
The following lemma provides a way to avoid ramification of factors of $q$ while taking a $q$-th root.
\begin{lemma}
\label{notq}
If $K$ is a number field containing $\xi_q$, a $K$-prime $\qq_K$ is a factor of $q$ and 
\[
\ord_{\qq_K}(c - 1) \geq 3\ord_{\qq_K} q,
\]
 then $\qq_K$ splits completely in the extension $K(\sqrt[q]{c})/K$. 
\end{lemma}
\begin{proof}
By Corollary \ref{cor:q-thpower} the polynomial $X^q-c$ has a root in  $\qq_K$-adic completion of $K$, and since the field contains the primitive $q$-th root of unity, the polynomial has $q$ distinct roots.  Thus, the local degree is one for all the factors  above $\qq_K$.
\end{proof}

The next two propositions explain the purpose of introducing extension 
\[
L =K(\sqrt[q]{1+ x^{-1}}, \sqrt[q]{1+ (bx^q+b^q)^{-1}},\sqrt[q]{1+(c+c^{-1})x^{-1}}).
\]
  In $L$:
\be
\item  all primes that are zeros of $x$ and $bx^q+b^q$ ramify unless the order of these zeros is divisible by $q$;
\item  all primes that are zeros and poles of $c$  ramify unless the order of $c$ at these primes is divisible by $q$;
\item  we  avoid ramifying primes in the cyclic extension obtained taking the $q$-th root of $c$, where we are going solve norm equations; 
\item  we make sure that zeros of $x$ do not have any influence on whether the norm equation has solutions.
\ee
\begin{proposition}
\label{prop:badprime}
If $K$ is a number field containing $\xi_q$, and for some elements $b, c \in K$ and some $K$-prime $\pp_K$ the following assumptions are true:
\be
\item \label{ass:1} $\pp_K$ is not a factor of $q$,
\item \label{ass:2} $c$ is not a $q$-th power modulo $\pp_K$  (note that this assumption includes the assumption that  $\ord_{\pp_{K}}c = 0$),
\item \label{ass:3}$\ord_{\pp_K}x <0$,
\item \label{ass:4}$\ord_{\pp_K}b \not \equiv 0 \mod q, $
\item \label{ass:5} $q\ord_{\pp_K}x < (q-1)\ord_{\pp_K}b$,
\ee
then for every prime factor $\pp_{L}$ of $\pp_{K}$ in $L$ we have that
\be
\item $\ord_{\pp_L}x <0$,
\item $c$ is not a $q$-th power modulo $\pp_L$ and thus not a $q$-th power in $L$,
 and
\item $\ord_{\pp_L}(bx^q + b^q)\not \equiv 0 \mod q$.
\ee
\end{proposition}
\begin{proof}
First  by properties of primes and Assumption \ref{ass:3}, we have that  $\ord_{\pp_L}x <0$.  By Assumption \ref{ass:4}, we have that $\ord_{\pp_{K}}b\not \equiv 0 \mod q$. Next we note that $\ord_{\pp_{K}}(x^{-1}) >0$,
 and therefore by Lemma \ref{le:first}, Part \ref{firstit:3}  we have that $\pp_{K}$ splits completely into distinct factors in the extension $L_1/K$. (We remind the reader that $L_1= K(\sqrt[q]{1+ x^{-1}})$.)  Thus, in $L_1$ we have that $\ord_{\pp_{L_1}}x <0$, $\ord_{\pp_{L_1}}b \not \equiv 0 \mod q$, and $c$ is not a $q$-th power modulo $\pp_{L_1}$.  We now note that by Assumption \ref{ass:5} we have that $q\ord_{\pp_K} x+\ord_{\pp_K}b < q\ord_{\pp_K}b$,  and therefore
 \[
 \ord_{\pp_{L_1}}(bx^q+b^q) =\ord_{\pp_{L_1}}b +q\ord_{\pp_{L_1}}x <0.
 \]
  Further, by Assumption \ref{ass:4} we have that $\ord_{\pp_{L_1}}(bx^q+b^q) \not \equiv 0 \mod q$.  Applying Lemma  \ref{le:first}, Part \ref{firstit:3} again, this time over the field $L_2 =L_1(\sqrt[q]{1+ (bx^q+b^q)^{-1}})$, we see that in the extension $L_2/L_1$, the $L_1$- prime $\pp_{L_1}$ splits completely into distinct factors and thus $c$ is not a $q$-th power modulo any $\pp_{L_2}$, while  $\ord_{\pp_{L_2}}(bx^q+b^q) \not \equiv 0 \mod q$ and  $\ord_{\pp_{L_2}}(bx^q+b^q) <0.$   Since, by assumption, $\ord_{\pp_{K}}c =0$ and therefore $\ord_{\pp_{L_2}}c =0$, by Lemma \ref{le:first}, Part \ref{firstit:3} one more time, $\pp_{L_2}$ will split completely into distinct factors in the extension $L/L_2,$ and,  as before, this would imply that $c$ is not a $q$-th power in $L$ or modulo any $\pp_L$ above $\pp_{K}$.  Here we remind the reader that  $L =L_2(\sqrt[q]{1+(c+c^{-1})x^{-1}}).$
  Finally, we also have  $\ord_{\pp_L}(bx^q+b^q) \not \equiv 0 \mod q.$
\end{proof}
\begin{proposition}
\label{prop:fixorder}
 If $K$ is a number field containing $\xi_q$,  and $x,c, b \in K, L$ are as in Proposition \ref{prop:badprime}, then for any $L$-prime $\aaa_L$  that is not a factor of $q$ and is not a pole of $x$, the following statements hold:
\be
\item $\ord_{\aaa_L}c \equiv 0 \mod q$;
\item $\ord_{\aaa_L}(bx^q+b^q) \equiv 0 \mod q$;
\item $\ord_{\aaa_L}x \equiv 0 \mod q$.
\ee
\end{proposition}
\begin{proof}
We again proceed by applying Lemma \ref{le:first} three times.  In the extension $L_1/K$, where $L_1= K(\sqrt[q]{1+ x^{-1}})$,  all the primes that are zeros of $x$ of order not  divisible by $q$ are ramified by Lemma \ref{le:first}, Part \ref{firstit:4}, since for any $K$-prime $\aaa_{K}$ such that $\ord_{\aaa_{K}}x>0$ we have that $\ord_{\aaa_{K}}(1+x^{-1})=\ord_{\aaa_{K}}(x^{-1})<0$. 

  In the extension $L_2/L_1$, where  $L_2=L_1(\sqrt[q]{1+ (bx^q+b^q)^{-1}})$, as before, we ramify all the primes $\aaa_{L_1}$ such that $\ord_{\aaa_{L_1}}(bx^q+b^q)>0$ and $\ord_{\aaa_{L_1}}(bx^q+b^q) \not \equiv 0 \mod q.$ Further, if $\aaa_{L_1}$ is a pole of $bx^q+b^q$ but not a pole of $x$, then it is a pole of $b$ and therefore $\ord_{\aaa_{L_1}}(bx^q+b^q)  = q\ord_{\aaa_{L_1}}b$.
  
  Finally,  $(c+c^{-1})x^{-1}$ has poles at all primes occurring in the divisor of $c$ and not poles of $x$.  Since in $L_1$, and therefore in $L_2$, all zeros of $x$ are of order divisible by $q$, if $c$ has a pole or a zero of degree not divisible by $q$, and the prime in question  is not a pole of $x$, it follows that $(c+c^{-1})x^{-1}$ has a pole of degree not divisible by $q$ at this prime, forcing it to ramify  in the extension $L_2(\sqrt[q]{1+(c+c^{-1})x^{-1}})/L_2$.  Thus, $\ord_{\aaa_L}c \equiv 0 \mod q$ for any prime $\aaa_L$ not dividing $q$ and not a pole of $x$.
\end{proof}

We now consider what happens to factors of $q$ under cyclic extensions of degree $q$.
\begin{proposition}
\label{prop:badprimeq}
If for some elements $x, d, a$ of a number field $K$ containing $\xi_q$ and some $K$-prime $\qq_K$ the following assumptions are true:
\be
\item \label{ass:1} $\qq_K$ is  a factor of $q$,
\item \label{ass:2} $\qq_K$ does not split in the extension $K(\sqrt[q]{a})/K$,
\item \label{ass:3}$\ord_{\qq_K}x <0$,
\item \label{ass:4}$\ord_{\qq_K}d \not \equiv 0 \mod q, $
\item \label{ass:5} $\ord_{\qq_K}d \leq -3\ord_{\qq_K}q$,
\item \label{ass:6} $\ord_{\qq_K}a=0$,
\item \label{ass:last} $q\ord_{\qq_K}x < (q-1)\ord_{\qq_K}d$,
\ee
then for every prime factor $\qq_F$ of $\qq_{K}$ in $F$ we have that
\be
\item $\ord_{\qq_F}x <0$,
\item  $\qq_{F}$ does not split in the extension $F(\sqrt[q]{a})/F$,
 and
\item $\ord_{\qq_F}(dx^q + d^q)\not \equiv 0 \mod q$.
\ee
\end{proposition}
\begin{proof}
First of all we note that 
$F=K(\sqrt[q]{1+ d^{-1}},\sqrt[q]{1+ (dx^q+d^q)^{-1}},\sqrt[q]{1+(a+a^{-1})d^{-1}}).$  Next we observe that over the $\qq_K$-adic completion $K_{\qq_K}$ of $K,$ a $q$-th root of $a$ generates an unramified extension of degree $q$.  Further, if  $G/K$ is a finite extension, where $\qq_K$ has a local degree one  (i.e. $e=f=1$)  factor $\qq_G$, then $G_{\qq_G}\cong K_{\qq_K},$ and   a $q$-th root of $a$ generates an unramified extension of degree $q$ over $G_{\qq_G},$ where $\qq_G$ does not split. 

 Now note that by Assumption \ref{ass:5}, we have that  $\ord_{\qq_K}d \leq -3\ord_{\qq_K}q$, and therefore by Corollary \ref{cor:q-thpower}  we have that $\qq_{K}$ splits completely into distinct factors in the extension $F_1/K$. (We remind the reader that $F_1= K(\sqrt[q]{1+ d^{-1}})$.)  Thus, in $F_1$ we have that $\ord_{\qq_{F_1}}x <0$, $\ord_{\qq_{F_1}}d \not \equiv 0 \mod q$ and $\qq_{F_1}$ has a factor of relative degree $q$  in the extension generated by adjoining $\sqrt[q]{a}$ to  $F_1$  for any $\qq_{F_1} \in \calC_{F_1}(\qq_K)$.  Further, by Assumption \ref{ass:last},
\[
q\ord_{\qq_K} x+\ord_{\qq_K}d < q\ord_{\qq_K}d \leq-3q\ord_{\qq_K}q ,
\]
and therefore
 \[
 \ord_{\qq_{F_1}}(dx^q+d^q) =\ord_{\qq_{F_1}}d +q\ord_{\qq_{F_1}}x<-3q\ord_{\qq_K}q <0.
 \]
  Further, by Assumption \ref{ass:4} we have that $\ord_{\qq_{F_1}}(dx^q+d^q) \not \equiv 0 \mod q$.  Applying Corollary \ref{cor:q-thpower} again, this time over the field $F_2 =F_1(\sqrt[q]{1+ (dx^q+d^q)^{-1}})$, we see that in the extension $F_2/F_1$, the $F_1$- prime $\qq_{F_1}$ splits completely into distinct factors. Consequently, any $\qq_{F_2}$ has  a factor of relative degree $q$  in the extension generated by adjoining $\sqrt[q]{a}$ to  $F_2$, while  $\ord_{\qq_{F_2}}(dx^q+d^q) \not \equiv 0 \mod q$ and  $\ord_{\qq_{F_2}}(dx^q+d^q) <0.$   
  
  Since, by assumption, $\ord_{\qq_{K}}a =0$ and therefore $\ord_{\qq_{F_2}}a =0$, by Corollary \ref{cor:q-thpower} one more time, $\qq_{F_2}$ will split completely into distinct factors in the extension $F/F_2,$ (here we remind the reader that  $F=F_2(\sqrt[q]{1+(a+a^{-1})d^{-1}})$) and,  as before, this would imply that any $\qq_F$ will have a factor of relative degree $q$  in the extension generated by adjoining $\sqrt[q]{a}$ to  $F$, while  $\ord_{\qq_F}(dx^q+d^q) \not \equiv 0 \mod q.$ So in particular, $a$ is not a $q$-th power in $F$. 
  
\end{proof}

Now a $q$-``analog'' of Proposition \ref{prop:fixorder}.
\begin{proposition}
\label{prop:fixorderq}
Under assumptions of Lemma \ref{prop:badprimeq},   for any $F$-prime $\aaa_F$  that is not a pole of $d$ and is not a pole of $x$, the following statements hold:
\be
\item $\ord_{\aaa_F}d\equiv 0 \mod q$;
\item $\ord_{\aaa_F}a \equiv 0 \mod q$;
\item $\ord_{\aaa_F}(dx^q+d^q) \equiv 0 \mod q$.
\ee
\end{proposition}
\begin{proof}
We again proceed by applying Lemma \ref{le:first} three times.  In the extension $F_1/K$, where $F_1= K(\sqrt[q]{1+ d^{-1}})$,  all primes that are  zeros of $d$ not of order divisible by $q$ are ramified by Lemma \ref{le:first}, Part \ref{firstit:4}, since for any $K$-prime $\aaa_{K}$ such that $\ord_{\aaa_{K}}d>0$ we have that $\ord_{\aaa_{K}}(1+d^{-1})<0$. 

  In the extension $F_2/F_1$, where  $F_2=F_1(\sqrt[q]{1+ (dx^q+d^q)^{-1}})$, as before, we ramify all the primes $\aaa_{F_1}$ such that $\ord_{\aaa_{F_1}}(dx^q+d^q)>0$ and
$\ord_{\aaa_{F_1}}(dx^q+d^q) \not \equiv 0 \mod q.$
 Further, if $\aaa_K$ is a pole of $dx^q+d^q$  but $\aaa_{K}$ is not a pole of $d$, then $\ord_{\aaa_K}(dx^q+d^q)  = q\ord_{\aaa_K}x$, and if for some pole $\qq_K$ of $d$ we have that $\ord_{\qq_K}x >0$, then $\ord_{\qq_K}(dx^q+d^q)  = q\ord_{\qq_K}d$.

  Finally,  $(a+a^{-1})d^{-1}$ has poles at all primes occurring in the divisor of $a$ and not poles of $d$.  Further in $F_2$ all zeros of $d$ are of orders divisible by $q$.  Thus if $a$ has a pole or a zero of degree not divisible by $q$,  it follows that $(a+a^{-1})d^{-1}$ has a pole of degree not divisible by $q$ at this prime, forcing it to ramify  in the extension $F_2(\sqrt[q]{1+(a+a^{-1})d^{-1}})/F_2$.  Thus, $\ord_{\aaa_F}a \equiv 0 \mod q$ for any prime $\aaa_F$ as described in the statement of the proposition.
\end{proof}

The lemma below considers some archimedean completions of a number field.
\begin{lemma}
\label{le:real}
If $c \in \Omega_2(K)$ and  $M = K(\sqrt{c})$, then any archimedean completion of $M$ is isomorphic to the corresponding archimedean completion of $K$.
\end{lemma}
\begin{proof}
Let $\sigma$ be an embedding of $M$ into $\tilde \Q$.    If $\sigma(M) \subset \tilde \Q \cap \R$, then the archimedean completion of $\sigma(M)$ is isomorphic to $\R$,  and the completion is isomorphic to $\C$ otherwise.  Therefore to prove the lemma, it is enough to show that whenever $\sigma(K) \subset \tilde \Q \cap \R$, we also have $\sigma(M) \subset \tilde \Q \cap \R$.   This implication follows from the fact that whenever $\sigma(K) \subset \tilde \Q \cap \R$, we have $\sigma(c) > 0$ and therefore $\sqrt{\sigma(c)} \in \R$.
\end{proof}
We will need the two lemmas below when analyzing what happens to factors of $q$ in number field extensions of degree $q$.
\begin{lemma}%
\label{le:existprime}
If $U/K$ is a Galois extension of number fields, $F/U$ is a cyclic number field extension, and the extension
$F/K$ is Galois, then there are infinitely many primes of $U$ not splitting in the extension $F/U$ and lying above a prime of $K$ splitting completely in $U$.
\end{lemma}%
\begin{proof}%
If $\sigma$ is a generator of $\Gal(F/U)$, then any prime of $F$ whose Frobenius  over $K$ is
\[
\sigma \in \Gal(F/U) \subset  \Gal(F/K)
\]
 will have the desired property.  Now Tchebotarev Density Theorem tells us that there are infinitely many such primes.
\end{proof}%

\begin{lemma}%
\label{le:frob}%
Let $F/U$ be a cyclic extension of number number fields such that for some rational prime $q$ we have that $[F:U]\equiv 0 \mod q^m$. Let
$N$ be the unique subfield of $F$ containing $U$ such that $[N:U]=q^m$. Let $\pp_F$ be a prime of $F$ and let $\pp_U$ be
the $U$-prime below it.  If  $\sigma$ is  the Frobenius automorphism of $\pp_F$ and $\sigma$ is not a $q$-th power in $\Gal(F/U)$, then $\pp_U$ does not split in the extension $N/U$.
\end{lemma}%
\begin{proof}%
Observe that $\Gal(F/N)$ is the set of all elements of the Galois group that are $q^m$-th powers.  Thus, since $\sigma$ is not a $q$-th power in $\Gal(F/U)$, we must have that $q^m$ is the smallest positive power $r$ of $\sigma$ such that $\sigma^r \in \Gal(F/N)$.  Therefore, we have that  $\sigma_{|N}$ has order $q^m$ and thus generates the Galois group of $N$ over $U$.  Hence, the decomposition group of $\pp_F \cap N=\pp_N$ is the Galois group of $N/U$, and $\pp_U$ does not split in the extension $N/U$.
\end{proof}%
We now construct a cyclic extension of degree equal to a power of $q$ where  $q$ can have an arbitrarily high relative degree and no ramified factors.
\begin{lemma}
\label{le:qextensions}%
If $q$  is a rational prime, $m \in\Z_{>0}$, then there exists a totally real cyclic extension of $\Q$ of degree $q^m$ where  $q$ does not split.
\end{lemma}%
\begin{proof}%
Let $\ell$ be a rational prime satisfying the following conditions:
\be%
\item $\ell$ splits completely in    $\Q(\xi_{q^m})/\Q$. %
\item Factors of $\ell$ in $\Q(\xi_{q^m})$ do not split in the extension $\Q(\xi_{q^m}, \sqrt[q]{q})/\Q(\xi_{q^m})$.
\ee%
(Observe that by Lemma \ref{le:existprime} there are infinitely many such $\ell$'s.)  It follows that  $\ell\equiv 1\mod q^m$, but $q$ is not a $q$-th power mod $\ell$.  Indeed, since both bases $\{1, \xi_q,\ldots, \xi_q^{(q-1)q^{m-1}}\}$ and $\{1, \sqrt[q]{q},\ldots, \sqrt[q]{q^{q-1}}\}$ are integral bases with respect to $\ell$ and all of its factors,  the factorization of $\ell$ and its factors in the extensions $\Q(\xi_{q^m})/\Q$  and $\Q(\xi_{q^m}, \sqrt[q]{q})/\Q(\xi_q)$ corresponds to the factorization of the respective minimal polynomials modulo $\ell$.  Consequently, $\Z/\ell$ contains a $q^m$-th root of unity, so that $q^m | (\ell-1)$, and the polynomial $T^q-q$ has no roots modulo any factors $\ell$ in $\Q(\xi_q)$.  

  Now consider
the extension $\Q(\xi_{\ell})/\Q$ and note that it is  of degree divisible by $q^m$.   If $\tau$ is the Frobenius of $q$, then $\tau(\xi_{\ell})=\xi_{\ell}^{q}$ and $\tau$ is not a $q$-th power in $\Gal(\Q(\xi_{\ell})/\Q)$.  Indeed, suppose $\tau = \sigma^q$ for some $\sigma \in  \Gal(\Q(\xi_{\ell})/\Q)$.  Let $r$ be a positive integer such that $\sigma(\xi_{\ell})=\xi_{\ell}^r$ and therefore $\xi_{\ell}^q=\tau(\xi_{\ell})=\sigma^q(\xi_{\ell})=\xi_{\ell}^{r^q}$ implying $q \equiv r^q \mod \ell$ in contradiction of our assumption on $\ell$ and $q$.  Therefore, by Lemma \ref{le:frob}, we conclude that  $q$ will not split in the unique degree $q^m$ extension of $\Q$ contained in $\Q(\xi_{\ell})$.  
\end{proof}
We now use the lemma above to construct a cyclic extension of a number field where $q$ has relative degree $q$ and no ramified factors. We do this in two steps.  The first step is the lemma below.
\begin{lemma}
\label{le:staycyc}
If $G$ is algebraic over $\Q$, $H$ a number field with $H/\Q$ cyclic, then $GH/G$ is cyclic with $[GH:G] | [H:\Q]$.
\end{lemma}
\begin{proof}
If $A=G\cap H$, then, since $H/\Q$ is Galois,  $[H:A]=[GH:G]$ and thus   $[GH:G]$ divides $ [H:\Q].$  Indeed, let $\alpha \in H$ generate $H$ over $\Q$ and therefore also $GH$ over $G$, and let $a_0 +a_1T + \ldots +T^r$ be the monic irreducible polynomial of $\alpha$ over $G$.  Since all the conjugates of $\alpha$ over $\Q$ are in $H$, all the conjugates of $\alpha$ over $G$ are in $H$, and thus $a_0,\ldots, a_{r-1} \in H$ and hence in $A$.  So the degree of $\alpha$ over $G$ is at least as large as the degree of $\alpha$ over $A$.  Since $A \subseteq G$, these degrees must be equal.  

  Further $H/A$ is again a cyclic extension, and all the  the conjugates of $\alpha$ over $A$ and over $G$ are the same.  Hence, $\Gal(GH/G) \cong \Gal(H/A)$ and we can conclude that the extension $GH/G$ is cyclic.
\end{proof}

This is the second step of our construction.
\begin{lemma}
\label{le:notsplit}
Let $G$ be a number field such that for some prime $\pp_G$ of $G$ lying above a rational prime $\pp_{\Q}$ we have that $\ord_q(f(\pp_G/\pp_{\Q}))=m$.  Suppose now that $H$ is a cyclic extension of $\Q$ of degree $q^r$ with $r >m$, where $\pp_{\Q}$ does not split.  Let $GH$ be the field compositum of $G$ and $H$ inside the chosen algebraic closure of $\Q$.  Under these assumptions, there exists a field $\hat G$ such that $G \subseteq \hat G \subset GH$ and $GH/\hat G$ is a cyclic extension  of degree $q$ where no factor of $\pp_H$ splits.
\end{lemma}
\begin{proof}
Consider the following field diagram
\begin{center}
\xymatrix{{\pp_{GH} \in GH}& {\pp_G \in G}\ar[l]\\
{\pp_H\in H}\ar[u]&{\pp_\Q \in\Q}\ar[l]\ar[u]}
\end{center}
and observe that $f(\pp_{GH}/\pp_\Q) \geq q^r$, while $\ord_q(f(\pp_G/\pp_{\Q})) =m <r$.  Consequently,
\[
\ord_q(f(\pp_{GH}/\pp_G))>1
\]
and thus $f(\pp_{GH}/\pp_G)>1$.  By Lemma \ref{le:staycyc}, the extension $GH/G$ is cyclic of degree that is a power of $q$.  Further, by Proposition 8, of Chapter II, \S4 of \cite{L}, $GH/G$ is unramified at all the factors of $\pp_G$. Let $\sigma$ be a generator of the $\Gal(GH/G)$ and observe that for some positive integer $i$,  the Frobenius automorphism of any factor $\pp_{GH}$ of $\pp_G$ over $G$ is $\sigma^i \not = \id$ and must be of order divisible by $q$.  Now, if $\hat G \not =GH$ is the fixed field of $\sigma^{\ord \sigma^i/q}$, we have that any factor $\pp_{\hat G}$ of $\pp_G$ in $\hat G$ will not split in the extension $GH/\hat G$ and $[GH:\hat G]=q$.
\end{proof}
Since for any $G$ and $H$ as above,  the field $\hat G$ satisfying  $G \subset \hat G \subset GH$ and $[GH:\hat G]=q$ is unique, we have the following corollary.
\begin{corollary}
\label{cor:notsplit}
Let $G, H$ be as in Lemma \ref{le:notsplit}, and assume additionally that for {\it any} $G$-prime  $\pp_G$ lying above a rational prime $\pp_{\Q}$ we have that $\ord_qf(\pp_G/\pp_{\Q})<[H:\Q]$.  Let $\hat G$ be a subfield of $GH$ such that $G \subset \hat G$ and $[GH:\hat G]=q$.  In this case no $\hat G$-factor of $\pp_\Q$ splits in the extension $GH/\hat G$.
\end{corollary}
We now consider the case when $q=2$ and examine generators of $GH$ over $\hat G$.
\begin{lemma}
\label{makereal}
Let $G, \hat G, H$ be as in  Corollary \ref{cor:notsplit}, let $q=2$, and assume $H$ is totally real.  Suppose $HG =\hat{G}(\sqrt{a})$, $a \in \hat{G}$.  In this case, if $\sigma : \hat G \longrightarrow \tilde \Q \cap \R$ is an embedding of $\hat G$, then $\sigma(a)>0$.
  \end{lemma} 
  \begin{proof}
  Since $H$ is totally real, for any embedding $\sigma: HG \longrightarrow \tilde \Q$, we have that 
  \[
  \sigma(HG) \subset \R \Leftrightarrow \sigma(G) \subset \R. 
  \]
   If $\sigma(\hat G) \subset \R$, then $\sigma(G) \subset \R$ and $\sigma(HG) \subset \R$ implying $\sqrt{\sigma(a)} \in \R$ and $\sigma(a) \geq 0$. 
  
  \end{proof}

\section{Local degree in infinite extensions.}
\label{primes}
Before proceeding with definitions of integers and integral closures of rings of $\calS$-integers in infinite algebraic extensions of $\Q$, we would like to discuss and clarify the conditions we will use. We start with adding notation.
\begin{notation}
\label{not:1.1}
\begin{enumerate}
\item Let $K_{\inff}$ be an infinite algebraic extension of a number field $G$.
\item Let $I_G=I(G, K_{\inff})=\{K| K \mbox{ is a number field  such that } G \subseteq K \subset K_{\inff}\}$.
\item For any $M\in I_G$, let
\[
I_M=I_M(G, K_{\inff})=\{K| K \mbox{ is a  number field  such that } G \subseteq M \subseteq K \subset K_{\inff}\}.
\]
\item \label{J} For any $M  \in I_G$, let $J_M(G, K_{\inff})$ be an ordered by inclusion subset of $I_M$ such that the union of all the fields in $J_M$ is $K_{\inff}$.  If $\pp_M$ is a prime of $M$, then prime factors of $\pp_M$ in the fields of $J_M$ generate a tree.  A path in such a tree corresponds to a prime ideal of $O_{K_{\inff}}$--the ring of integers of $K_{\inff}$.  We will refer to $J_M$ as a field path from $M$ to $K_{\inff}$.
\item If $M \in I_G$ and $\calS_M$  is a set of primes of $M$, then let $O_{K_{\inff},\calS_{K_{\inff}}}$ denote the integral closure of $O_{M,\calS_M}$ in $K_{\inff}$.  As mentioned above, $O_{K_{\inff}}$ will denote the ring of algebraic integers of $K_{\inff}$
\item If $M$ is a number field, $\pp_M$ is a prime of $M$, and $K\in I_M$, then let $\calC_K(\pp_M)$, as above,  denote the set of all prime factors of $\pp_K$ in $M$.  Let $\calC_{\inff}(\pp_M)=\bigcup_{K \in I_M}\calC_K(\pp_M)$.
\ee
\end{notation}
We now describe conditions on  the primes necessary for our definitions of integers.  The two diagrams below correspond to the Definition \ref{unbounded} of $q$-unbounded and completely $q$-bounded primes.
\begin{center}
{\bf A diagram for $q$-unbounded primes}\\
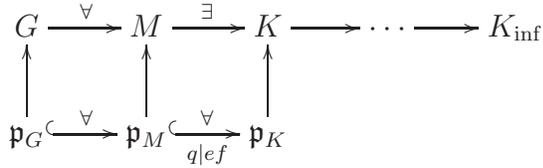
\xymatrix{
{G}\ar[r]^{\forall}&M \ar[r]^{\exists}& {K}\ar[r] &{\cdots}\ar[r] &{K_{\inff}}\\
{\pp_G}\ar[u]\ar@{^{(}->}[r]^{\forall}&{\pp_M}\ar[u]\ar@{^{(}->}[r]^{\forall}_{q | ef}&{\pp_K}\ar[u]
}
\end{center}

\begin{center}
{\bf A diagram for completely $q$-bounded primes}\\
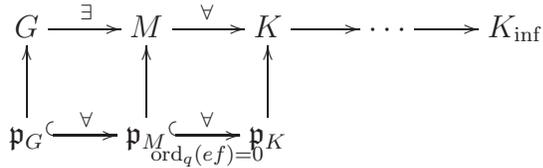
\xymatrix{
{G}\ar[r]^{\exists}&M \ar[r]^{\forall}& {K}\ar[r] &{\cdots}\ar[r] &{K_{\inff}}\\
{\pp_G}\ar[u]\ar@{^{(}->}[r]^{\forall}&{\pp_M}\ar[u]\ar@{^{(}->}[r]^{\forall}_{\ord_q(ef)=0}&{\pp_K}\ar[u]
}
\end{center}

\begin{definition}[$q$-unbounded and $q$-bounded  primes]
\label{unbounded}
Let $q$ be a rational prime and let $\pp_G$ be a prime of $G$ satisfying the following condition:   for any $M \in I_G$ there exists $K \in I_M$  such that for any  $\pp_M \in \calC_M(\pp_G)$ and any $\pp_K$ in $\calC_K(\pp_M)$ we have that  
\[
d(\pp_K/\pp_M)=e(\pp_K/\pp_M)f(\pp_K/\pp_M) \equiv 0 \mod q,
\]
 where as usual $e(\pp_K/\pp_M)$ is the ramification degree of $\pp_K$ over $\pp_M$, $f(\pp_K/\pp_M)$ is the relative degree of $\pp_K$ over $\pp_M$, and $d(\pp_K/\pp_M)$ is the local degree of $\pp_K$ over $\pp_M$.  In this case we call $\pp_G$ {\it $q$-unbounded}.  (See the diagram above.)

If there exists $M \in I_G$ such that for any $K \in I_M$,   for any  $\pp_M \in \calC_M(\pp_G)$, and any $\pp_K$ in $\calC_K(\pp_M)$ we have that   $\ord_qd(\pp_K/\pp_M)=0$, we call $\pp_G$ {\it completely $q$-bounded}.  (See a diagram above.)   

If $\pp_G$ is {\bf not} $q$-unbounded, we  call $\pp_G$ {\it $q$-bounded}.
If  every prime in $\calC_{\inff}(\pp_G)$ is $q$-bounded, we call $\pp_G$ {\it hereditarily  $q$-bounded}.   

If every prime of $G$ is hereditarily $q$-bounded  in $K_{\inff}$, and all the factors of $q$ are completely $q$-bounded, then we will call $K_{\inff}$ itself $q$-bounded.

\end{definition}
     
Observe that if a prime is completely $q$-bounded, it is hereditarily $q$-bounded.   As is shown below, we  need all the primes of $G$ to be hereditarily $q$-bounded, and we need $q$ to be completely $q$-bounded for our definition method to work for the ring of integers.  At the same time the unbounded primes can be used to define ``big subrings''.  
 \begin{remark}
 \label{rem:unbounded}
   One can rephrase the definition of a $q$-unbounded prime as follows.  A prime $\pp_G$ of $G$ is unbounded if for every $n \in \Z_{>0}$ there exists a field $M \in I_G$ such that for any $\pp_M \in \calC_M(\pp_G)$ we have that $e(\pp_M/\pp_G)f(\pp_M/\pp_G)=d(\pp_M/\pp_G)\equiv 0 \mod q^n$, where $d(\pp_M/\pp_G)$ as above is the local degree $[M_{\pp_M}:G_{\pp_G}]$.
    \end{remark}    
    We also need the following definition.
    \begin{definition}
      Given a $G$-prime $\pp_G$ and a field path $J_G=\{G -M_1 -M_2 \ldots \}$ from $G$ to $K_{\inff}$,  as described in  Notation \ref{not:1.1}, Part \ref{J}, call a path $\calP=\{\pp_G-\pp_{M_1}-\pp_{M_2} \ldots \}$ through the tree of $\pp_G$-factors   $q$-bounded  if  there exists $i \in \Z_{>0}$ such that for all integer $j \geq i$ we have that  $\ord_q(d(\pp_{M_j}/\pp_G))=\ord_q(d(\pp_{M_i}/\pp_G))=n_i$.  Also call $M_i$ a  {\it $q$-bounding field}  and call $n_i$ a {\it $q$-bounding order}.   
     \end{definition}
     \begin{remark}
     \label{work off path}
     A $q$-bounding field and a $q$-bounding order  also ``work'' off the field path where they were defined.  Indeed, let $M$ and $n$ be a $q$-bounding field and order defined along some field path $J_G$, and let $N \in I_G$.  In this case for some $\pp_N \in \calC_N(\pp_G)$ it is true that $\ord_q(d(\pp_N/\pp_G)) \leq n$.  Indeed, some field $L$ along the field path $J_G$ contains $M$ and $N$ and for some $\pp_L \in \calC_L(\pp_G)$ we have that $\ord_q(d(\pp_L/\pp_G)) = n$.  Thus, for $\pp_N = \pp_L \cap N \in \calC_N(\pp_G)$ it is the case that $\ord_q(d(\pp_N/\pp_G))  \leq \ord_q(d(\pp_L/\pp_G)) = n$.  Similarly, for any $L \in I_M$ we have that for some $\pp_L \in \calC_L(\pp_M)$ it is the case that $\ord_qd(\pp_L/\pp_M)=0$.
          \end{remark}
    \begin{lemma}
    \label{bounded path}
Choose any field path $J_G$ as in Notation \ref{not:1.1}, Part \ref{J} and consider the corresponding tree of  factors for some prime $\pp_G$ of $G$.  We claim that 
   $\pp_G$ is $q$-bounded if and only if it lies along a $q$-bounded path.  
   \end{lemma}

    \begin{proof}
     Indeed, suppose $\pp_G$ is $q$-bounded and let $n \in \Z_{>0}$ be such that for any $M \in J_G$ for some $\pp_M \in \calC_M(\pp_G)$ we have that $d(\pp_M/\pp_G)\not \equiv 0 \mod q^n$.  From the tree of $\pp_G$ factors corresponding to $J_G$ remove all the ``nodes'' (i.e. factors of $\pp_G$) with the local degree with respect to $\pp_G$ divisible by $q^n$.  Note that if a  node survives removal, all of its predecessors must survive too.  Thus, the tree structure is preserved under the removal of the nodes with the local degree with respect to $\pp_G$ divisible by $q^n$.   This tree will have arbitrarily long paths and thus by K\"{o}nig's Lemma an infinite path.  Since the order at $q$ of the local degree along this path is bounded, after some point the degree can grow only by factors prime to $q$. 
     
     Conversely,  along a $q$-bounded path the order of the local degree at $q$ will be bounded and therefore we cannot have arbitrarily large powers of $q$ divide the local degree for all the factors of a prime on such a path.
     \end{proof}

  In the case a prime $\pp_G$ is completely $q$-bounded, by definition, there is a $q$-bounding field and a $q$-bounding order which work along all paths through the factor tree.
  \begin{definition}
  Let $\pp_G$ be a completely $q$-bounded prime and let $M \in I_G$  be such that for any $K \in I_M$, for any $\pp_M \in \calC_M(\pp_G)$, and any $\pp_K \in \calC_K(\pp_M)$ we have that $\ord_q{d(\pp_K/\pp_M)}=0$.  In this case call $M$ a {\it completely $q$-bounding field} (for $\pp_G$).  Call $\max_{\pp_M \in \calC_M(\pp_G)}(\ord_q(d(\pp_M/\pp_G)))$ a {\it completely $q$-bounding order} (for $\pp_G$).  
  \end{definition}

\section{Defining the Ring of Integers in  Infinite Extensions of $\Q$.}
\label{inf ext}
Our plan is to deal with all but finitely many primes first.  This is accomplished in the section below.
\subsection{The main part of the definition}
We will use the following notation and assumptions in this section.
\begin{notationassumptions}
\begin{itemize} 

\item Let $K_{\inff}$ be an infinite algebraic extension of $\Q$.
\item Let $G \subset K_{\inff}$ be a number field, and let $\calS_G$ be a finite, possibly empty set of primes of $G$.  Suppose all the primes of $G$ not dividing $q$ and not in $\calS_G$ are {\it heredeterily $q$-bounded}  in $K_{\inff}$. 
\item Let $\calQ_G$ be the set of all factors of $q$ in $G$.
\item Let $\calW_G =  \calS_G \cup \calQ_G$
\item Let $O_{K_{\inff},\calW_{K_{\inff}}}, O_{K_{\inff},\calS_{K_{\inff}}}, O_{K_{\inff},\calQ_{K_{\inff}}},$ denote the integral closure of $O_{G,\calW_G}$,  $O_{G,\calS_G}$ and $O_{G,\calQ_G}$ respectively in $K_{\inff}$.
\end{itemize}
\end{notationassumptions}
\begin{proposition}
\label{prop:norm}
If $\xi_q \in G$, $b, x  \in K_{\inff}$, $x \not=0, bx^q+b^q \not =0$
\[
c \in \Omega_q(K_{\inff})  \cap \Phi_q(K_{\inff}) \cap \Theta_q(K,\calS_{K_{\inff}}),
\]
   and there exists $y \in L_{\inff}$, where $L_{\inff}=K_{\inff}(\sqrt[q]{1+ x^{-1}}, \sqrt[q]{1+ (bx^q+b^q)^{-1}},\sqrt[q]{1+(c+c^{-1})x^{-1}})$ such that
\begin{equation}
\label{eq:norminside}
{\mathbf N}_{L_{\inff}(\sqrt[q]{c})/L_{\inff}}(y) = bx^q+b^q,
\end{equation}
then there exists a field $M\in I_G$ such that  for any field $K \in I_{M}$, for any non-archimedean prime $\pp_K$ of $K$ not in $\calW_K$, it is the case that one of the following conditions holds:
\be
\item \label{nass:1}$c$ is a $q$-th power mod $\pp_K$, or
\item \label{nass:2}$\ord_{\pp_K} x \geq 0$, or
\item \label{nass:3}$q\ord_{\pp_K}x  \geq (q-1)\ord_{\pp_K}b$, or
\item \label{nass:4}$\ord_{\pp_K}b \equiv 0 \mod q$.
\ee
At the same time, if $x \in O_{K_{\inff},\calW_{K_{\inff}}}$,  then \eqref{eq:norminside} has a solution $y \in L_{\inff}$.
\end{proposition}
\begin{proof}
Suppose that  \eqref{eq:norminside} holds for some $x, b,c,y$ as specified above.  Let $M \in I_G$ be such that 
\begin{equation}
\label{eq:cond1}
x , b, c \in M, 
\end{equation}
\begin{equation}
\label{eq:cond1.1}
y \in L_{M}(\sqrt[q]{c}), \mbox{ where } L_{M}=M(\sqrt[q]{1+ x^{-1}}, \sqrt[q]{1+ (bx^q+b^q)^{-1}},\sqrt[q]{1+(c+c^{-1})x^{-1}})
\end{equation}
 and 
\begin{equation}
\label{eq:cond2}
[L_{\inff}(\sqrt[q]{c}):L_{\inff}]=[L_M(\sqrt[q]{c}):L_M].
\end{equation}
 In this case, for any $K \in I_M$, we also have that  $x ,b, c \in K$, $y \in L_K(\sqrt[q]{c}),$ where 
\[
L_{K}=K(\sqrt[q]{1+ x^{-1}}, \sqrt[q]{1+ (bx^q+b^q)^{-1}},\sqrt[q]{1+(c+c^{-1})x^{-1}})
\]
with
\[
[L_{\inff}(\sqrt[q]{c}):L_{\inff}]=[L_K(\sqrt[q]{c}):L_K],
\]
 and therefore it is also the case 
\begin{equation}
\label{eq:norminsidefinite}
{\mathbf N}_{L_K(\sqrt[q]{c})/L_K}(y) = bx^q+b^q.
\end{equation}

 Now, if for some $K$-prime $\pp_K$ such that $\pp_K \not \in \calW_K$,  we have that none of the Conditions \ref{nass:1} -- \ref{nass:4} is satisfied,  then by Proposition \ref{prop:badprime}, we have that 
\[
\ord_{\pp_{L_K}}(bx^q+b^q)\not \equiv 0 \mod q
\]
 and $c$ is not $q$-th power modulo $\pp_{L_K}$.  Hence by by Lemma \ref{le:cycextensions} we conclude that the norm equation \eqref{eq:norminsidefinite} has no solution in $ L_K(\sqrt[q]{c})$ contradicting our assumptions.   
 
Suppose now that $x \in O_{K_{\inff},\calW_{K_{\inff}}}$, let $M \in I_G$ satisfy assumptions \eqref{eq:cond1},  \eqref{eq:cond2} and be such that $c \in \Omega_q(M)   \cap \Phi_q(M) \cap \Theta_q(M,\calS_{M})$.  (We can find $M$ satisfying $c \in \Omega_q(M)$ by Remark \ref{rem:omega}.)  We now choose any $K \in I_M$ and show that  \eqref{eq:norminsidefinite} has a solution  $y \in L_K(\sqrt[q]{c})$.  Since \eqref{eq:cond2} insures that for any $y \in L_K(\sqrt[q]{c})$, it is the case that 
\begin{equation}
\label{eq:samenorm}
{\mathbf N}_{L_K(\sqrt[q]{c})/L_K}(y) ={\mathbf N}_{L_{\inff}(\sqrt[q]{c})/L_{\inff}}(y),
\end{equation}
we need to solve ${\mathbf N}_{L_K(\sqrt[q]{c})/L_K}(y) =bx^q+b^q$ only.

Since $x \in O_{K_{\inff}, \calW_{K_{\inff}}}$, we have that $x \in O_{K,\calW_K}$.  Further, we also have that 
\[
c \in \Omega_q(K)  \cap \Phi_q(K) \cap \Theta_q(K,\calS_{K}),
\]
by definition of these sets and Remark \ref{rem:omega}.
In this case by Proposition \ref{prop:fixorder}, for  every prime $\aaa_{L_K}$, not dividing $q$ or any prime in $\calS_K$, we have the following:
\begin{itemize}
\item  $\ord_{\aaa_{L_K}}(bx^q+b^q) \equiv 0 \mod q$,
and
\item  $\ord_{\aaa_{L_K}}c \equiv 0 \mod q$.
\end{itemize}
Further, by Lemma \ref{notq} and by our assumption that $c \in \Phi_q(K)$, we know that factors of $q$ are not ramified in the extension $L_K(\sqrt[q]{c})/L_K$, and since the divisor of $c$ is a $q$-th power in $L_K$, the extension $L_K(\sqrt[q]{c})/L_K$ is unramified at all finite primes by Lemma \ref{le:first}.

 By  Hasse's Norm Principle (see Theorem 32.9 of \cite{Reiner}) this norm equation has solutions globally (i.e. in $L_K(\sqrt[q]{c})$) if and only if it has a solution locally (i.e. in every completion).   

 Observe further that locally every unit is a norm in an unramified extension (see Proposition 6, Section 2, Chapter XII of \cite{W}), and we do not have to worry about archimedean primes, given our assumption on $c$.   Indeed, if $q >2$, then $K\not \subset \R$ and therefore all the archimedean completions of all the fields involved are isomorphic to $\C$.  If $q=2$, then we have to worry about one possibility only: an archimedean completion of $L_K$ is isomorphic to $\R$, while a corresponding archimedean completion of $L_K(\sqrt{c})$ is isomorphic to $\C$.  However, this case is precluded by Lemma \ref{le:real} and our assumption that $c \in \Omega_q(K)$.

   Next we observe that since $L_K(\sqrt[q]{c})/L_K$ is a cyclic extension of prime degree, by Lemma \ref{le:cycextensions} every unramified prime either splits completely or does not split at all.  If a prime splits completely, then the local degree is one and every element of the field below is automatically a norm locally at this prime.  So the only primes where we might have elements which are not local norms are the primes which do not split, or, in other words, the primes where the local degree is $q$.   (Note that any factor of $q$ and any factor of a prime in $\calS_K$ split completely in the extension $L_K(\sqrt[q]{c})/L_K$ by our assumptions on $c$ and Lemmas \ref{le:first} and \ref{notq}.) 

So let $\rr_{L_K}$ be  a prime of local degree $q$ not in $\calW_{L_K}$.  By the argument above we have that $\ord_{\rr_{L_K}}(bx^q +b^q) \equiv 0 \mod q$.  In this case, by the Weak Approximation Theorem, there exists $u \in L_K$ such that $\ord_{\rr_{L_K}}u =1$  and therefore for some integer $m$ it is the case that  $u^{qm}(bx^q+b^q)$ has order 0 at $\rr_{L_K}$ or in other words $u^{qm}(bx^q+b^q)$ is a unit at $\rr_{L_K}$.   

As any $q$-th power of an $L_K$-element, $u^{mq}$ is a norm locally since the degree of the local extension is $q$ by our assumption.  Therefore, $u^{mq}(bx^q+b^q)$ is a norm at $\rr_{L_K}$ if and only if $(bx^q+b^q)$ is a norm at $\rr_{L_K}$.  But $u^{mq}(bx^q+b^q)$ is a unit at $\rr_{L_K}$ and therefore is a norm.  Hence $bx^q+b^q$ is a norm.
\end{proof}
\begin{corollary}
\label{def1}
If $\xi_q \in G$ then
\begin{equation}
\label{eq:A}
\begin{array}{c}
O_{K_{\inff},\calW_{K_{\inff}}}=\{0\} \cup\\
 \{x \in K_{\inff} \setminus \{0\}|\forall c \in \Theta_q(K_{\inff},\calS_{K_{\inff}}) \cap \Phi_q(K_{\inff})\cap \Omega_q(K_{\inff}) \forall b \in K_{\inff} \\ ((bx^q+b^q=0) \lor \exists y \in L_{\inff}(\sqrt[q]{c}): {\mathbf N}_{L_{\inff}(\sqrt[q]{c})/L_{\inff}}(y) = bx^q+b^q)\},
\end{array}
\end{equation}
In particular, if  $\calS_{K_{\inff}}$ is empty, then 
\begin{equation}
\label{eq:B}
\begin{array}{c}
O_{K_{\inff},\calQ_{K_{\inff}}}=\{0\} \cup\\
\{x \in K_{\inff}\setminus \{0\}|\forall c \in  \Phi_q(K_{\inff})\cap \Omega_q(K_{\inff}) \forall b \in K_{\inff}( (bx^q+b^q=0) \lor \exists y \in L_{\inff}(\sqrt[q]{c}): \\
{\mathbf N}_{L_{\inff}(\sqrt[q]{c})/L_{\inff}}(y) = bx^q+b^q)\},
\end{array}
\end{equation}
and if additionally $q >2$ or $K_{\inff}$ has no real embeddings, then 
\begin{equation}
\label{eq:C}
\begin{array}{c}
O_{K_{\inff},\calQ_{K_{\inff}}}=\{0\} \cup \\
\{x \in K_{\inff}\setminus \{0\}|\forall c \in  \Phi_q(K_{\inff}) \forall b \in K_{\inff} ((bx^q+b^q=0) \lor  \exists y \in L_{\inff}(\sqrt[q]{c}): {\mathbf N}_{L_{\inff}(\sqrt[q]{c})/L_{\inff}}(y) = bx^q+b^q)\}.
\end{array}
\end{equation}
\end{corollary}
\begin{proof}
 First we  assume $x \not \in O_{K_{\inff},\calW_{K_{\inff}}}$ and find $b,c$ as specified above for which \eqref{eq:norminside} has no solutions $y \in  L_{\inff}(\sqrt[q]{c})$.

 If  $x \not \in O_{K_{\inff},\calW_{K_{\inff}}}$, then for some prime $\pp_{G(x)} \not \in \calW_{G(x)}$ we have that 
\[
\ord_{\pp_{G(x)}}x<0,
\]
\[
 \pp_G=\pp_{G(x)}\cap G \not \in \calW_G
\]
 and $\pp_G$ is heredeterily $q$-bounded in  $K_{\inff}$.  Thus, $\pp_{G(x)}$ is $q$-bounded in $K_{\inff}$.  Let $M \in I_{G(x)}$ be a $q$-bounding field for $\pp_{G(x)}$ and note that by the Strong Approximation Theorem there exists $c \in \Theta_q(M,\calS_{M}) \cap \Phi_q(M)\cap \Omega_q(M) \subset \Theta_q(K_{\inff},\calS_{K_{\inff}}) \cap \Phi_q(K_{\inff})\cap \Omega_q(K_{\inff})  $ such that $c$ is not a $q$-th power modulo $\pp_M$, where $\pp_M \in \calC_M(\pp_{G(x)})$ lies along the $q$-bounded path for which $M$ is a $q$-bounding field.  Further, let $b \in M$ such that $\ord_{\pp_{M}}b=-1$ and thus $q\ord_{\pp_M}x  < (q-1)\ord_{\pp_M}b$.  Observe further that for any $K \in I_M$ we also have that
 \be
 \item $c \in \Omega_q(K)  \cap \Phi_q(K) \cap \Theta_q(K,\calS_{K})$, by definition of sets $\Omega_q(K), \Phi_q(K)$, and $\Theta_q(K,\calS_K)$,
 \item \label{it:same} for at least one $\pp_K \in \calC_K(\pp_M)$ we have that $d(\pp_K/\pp_M)$ and therefore $f(\pp_K/\pp_M)$ are not divisible by $q$ by definition of a $q$-bounding field, and therefore $c$ is not a $q$-th power modulo at least one $\pp_K \in \calC_K(\pp_M)$, 
 \item for the same $\pp_K$ as in \eqref{it:same} we also have that $e(\pp_K/\pp_M)$ is not divisible by $q$, and therefore $\ord_{\pp_K}b \not \equiv 0 \mod q$ while $q\ord_{\pp_K}x  < (q-1)\ord_{\pp_K}b$.
 \ee
 Thus none of Conditions \ref{nass:1} --\ref{nass:4} of Proposition \ref{prop:norm} is satisfied, and  hence \eqref{eq:norminside} has no solution $y \in L_{\inff}$.

\end{proof}

\subsection{Integrality at finitely many primes using complete $p$-boundedness for $p\not =q$.}
We now consider definitions of integrality at finitely many primes to define $\Theta_q(K_{\inff}, \calS_{K_{\inff}})$, $\Phi_q(K_{\inff})$  and their complements. One way to do this is to use a bit of  ``circular reasoning'' by introducing another rational prime $p$ into the picture and making additional assumptions about our field.  (Here ``circular reasoning'' refers to the fact that we use $q$ to define integrality at factors of $p$, and we use $p$ to define integrality at factors of $q$.)
\begin{notationassumptions}
\label{not:p}
\begin{itemize}
\item Let $p \not =q$ (with $q$ as above) be a rational prime.
\item Assume $\xi_p \in G$.  
\item Assume factors of $q$ and primes in $\calS_G$ are {\it completely $p$-bounded} in $K_{\inff}$ and are all prime to $p$. 
\item Let $\calW_G =\calS_G \cup\{ \mbox{factors of } q \mbox{ in }G\}$, as above.
 \item Let $M_p \in  I_G$ be a completely $p$-bounding field for {\it all} primes in $\calW_G$.  (Even though  completely bounding fields were defined  for a single prime, clearly any finite collection of completely bounded primes has a common completely bounding field, a field that contains a completely bounding field for each prime in the set.)
\end{itemize}

\end{notationassumptions}

\begin{proposition}
\label{finitely many for p not q}

Let  $d \in M_p$ be such that the denominator of its divisor  is divisible by every prime of $\calW_{M_p}$ and $d$ has  no other poles.  Assume further that for any $\pp_{M_p} \in \calW_{M_p}$ it is the case that $\ord_{\pp_{M_p}}d \not \equiv 0 \mod p$.
  (Note that such an element $d \in M_p$ exists by the Strong Approximation  Theorem.)
 Let  $a \in \Phi_p(M_p) \cap \Omega_p(M_p)$, and let $a$ be equivalent to a non-$p$-th power element of the residue field modulo any prime  of $\calW_{M_p}$.  (Existence of $a$ is also guaranteed by the Strong Approximation Theorem.)  Now let 
\[
N_{\inff}=K_{\inff}(\sqrt[p]{1+ d^{-1}},\sqrt[p]{1+ (dx^p+d^p)^{-1}},\sqrt[p]{1+(a+a^{-1})d^{-1}})
\]
 and let
\[
B(K_{\inff},p,a,d)=\{x \in K_{\inff}|\exists y \in N_{\inff}(\sqrt[p]{a}): {\mathbf N}_{N_{\inff}(\sqrt[p]{a})/N_{\inff}}(y) = dx^p+d^p\}.
\]
We claim $B(K_{\inff},p,a,d)=\{x \in K_{\inff}|\forall K \in I_{M_p(x)} \forall \pp_{K} \in \calW_{K}: \ord_{\pp_{K}}x > \frac{p-1}{p}\ord_{\pp_{K}}d \}$.
\end{proposition}
\begin{proof}
The proof of the proposition is almost identical to the proof of Proposition \ref{prop:norm}.  One should only point out the following two adjustments.
\be
\item By construction no pole of $d$ in any $K \in I_{M_p}$ occurs in the divisor of $a$, since $a$ is not a $p$-th power  modulo primes of $\calW_K$.  Thus, $(a+a^{-1})d^{-1}$ has poles at all the primes occurring in the divisor of $a$.  Also, all zeros of $d$ of orders not divisible by $p$ in $K$ are ramified with ramification degree $p$ before  we adjoin $\sqrt[p]{1+(a+a^{-1})d^{-1}}$, and therefore in $N_K=K(\sqrt[p]{1+ d^{-1}},\sqrt[p]{1+ (dx^p+d^p)^{-1}},\sqrt[p]{1+(a+a^{-1})d^{-1}})$ all zeros and poles of $a$ have order divisible by $p$.  
 
\item   For any prime $\pp_K \in \calW_K$  we have that  
\[
\ord_{\pp_K}(dx^p) \not = \ord_{\pp_K}(d^p),
\]
 since the left order is not equivalent to 0 mod $p$ and the right one is.  Thus under these circumstances, $\ord_{\pp_K}(dx^p+d^p) \equiv 0 \mod p$, implies that $\ord_{\pp_K}(dx^p+d^p)=\ord_{\pp_K}(d^p)$ and 
 \begin{equation}
 \label{eq:omodp}
 \ord_{\pp_K}x > \frac{p-1}{p} \ord_{\pp_K}d > \ord_{\pp_K}d. 
 \end{equation}
 Conversely, if  for some $K \in I_{M_p(x)}$ we have that  \eqref{eq:omodp} holds for all $K$-primes above primes of $\calW_G$, then $\ord_{\pp_K}(dx^p+d^p) \equiv 0 \mod p$ and $x \in B(K_{\inff},p,a,d)$.
  \ee
\end{proof}

We now use this definition of $B(K_{\inff},p,a,d)$ to obtain a definition of $R_{K_{\inff},\calW_{\inff}}$ -- the ring of elements of $K_{\inff}$ integral with respect to primes of $\calW_G$.  To do this we note the following.
\begin{lemma}
\label{integral}
$R_{K_{\inff},\calW_{\inff}}=\{ x \in  B(K_{\inff},p,a,d)| \forall y \in B(K_{\inff},p,a,d): xy \in B(K_{\inff},p,a,d)\}.$
\end{lemma}
\begin{proof}
First assume that $x \in R_{K_{\inff},\calW_{\inff}} \subset B(K_{\inff},p,a,d)$ and note that in this case $x$ has non-negative order at all primes of $\calW_{G(x)}$.  Thus, if for some field $K \in I_{G(x)}$ and some $K$-prime $\pp_K$ above a prime of  $\calW_G$ we have that $ \ord_{\pp_K}y > \frac{p-1}{p} \ord_{\pp_K}d$, then 
\[
 \ord_{\pp_K}xy \geq  \ord_{\pp_K}y> \frac{p-1}{p} \ord_{\pp_K}d.
 \]
    Conversely, suppose that $x \in B(K_{\inff},p,a,d) \setminus R_{K_{\inff},\calW_{\inff}}$ and note that in $K=M_p(x)$ we must have for some $K$-prime $\pp_K$ above a prime of $\calW_G$ that 
    \[
     \frac{p-1}{p} \ord_{\pp_K}d <\ord_p x <0.
  \]
    Therefore there exists an $r \in \Z_{\geq 1}$ such that $x^r \in B(K_{\inff},p,a,d)$ but $x^{r+1} \not \in B(K_{\inff},p,a,d)$.  Hence if we set $y =x^r$, we see that $y \in B(K_{\inff},p,a,d)$ but $xy \not \in B(K_{\inff},p,a,d)$.
\end{proof}
\subsection{Defining Integrality at finitely many primes  using complete $q$-boundedness.}
\label{onlyq}
Our next step is to show that we can get away using $q$-boundedness only (without introducing $p$-boundedness for an additional prime $p$).    The integrality at primes of $\calS_{K_{\inff}}$ can be handeled with complete $q$-boundedness only using sets $B(K_{\inff}, q, a, d)$ for appropriately selected $a$ and $d$ as above, since the primes of $\calS_K$ are not factors of $q$.  Thus we need to make special arrangements for factors of $q$ only.  Since we are going to use $q$-boundedness exclusively, we now drop Assumptions and Notation  \ref{not:p} and introduce the following assumptions and notation.
\begin{notationassumptions}
 We will use the following notation and assumptions.
\begin{itemize}
\item Assume all the primes of $\calW_G$ are completely $q$-bounded.  
\item Let $M_q$ be a completely $q$-bounding field for all primes  in $\calW_G$. 
\item Assume $\xi_q \in G$.
\item Let $\calQ_G$ be the set of all factors of $q$ in $G$. 
\item Let $f_{q}=\max_{\qq_{M_q} \in \calQ_{M_q}}\{f(\qq_{M_q}/q)\}$.  
\item Let $F/\Q$ be a totally real cyclic extension of degree $q^{f_q+1}$, where $q$ does not split. (Such an extension exists by Lemma \ref{le:qextensions}.)
\end{itemize}
\end{notationassumptions}
Now consider a cyclic extension $FK_{\inff}/K_{\inff}$ of degree $q^r$ (this extension is cyclic of degree equal to a power of $q$ by Lemma \ref{le:staycyc}), where $0 \leq r \leq f_q+1$.  We claim that in fact $r>0$.  Assume the opposite.  In this case for some $K \in I_{M_q}$ we have that $F \subseteq K$.  But the relative degree of any factor of $q$ in $K$ is at most $f_q$, while the relative degree of all the factors of $q$ in $FK$ is bigger than $f_q$.  Thus, $r>0$.  

Now let $E_{\inff}$ be the unique subfield of $FK_{\inff}$ such that $[FK_{\inff}:E_{\inff}]=q$ and $K_{\inff} \subset E_{\inff}$.  Since $\xi_q \in E_{\inff}$, we must have $FK_{\inff}=E_{\inff}(\sqrt[q]{a})$ for some $a \in E_{\inff}$  (this is so by Theorem 6.2, page 288 of \cite{Langalgebra}). Let $\beta \in E_{\inff}$ generate $E_{\inff}$ over $K_{\inff}$.  Now let $N \in I_{M_q}$ be such that $F \subset N(\sqrt[q]{a},\beta)$, $a \in N(\beta)$, and $\beta$ is of the same degree over $N$ as over $K_{\inff}$.  Let $K \in I_N$ and note that $\beta$ is of the same degree over $K$ as over $N$, $a \in K(\beta)$, and  $F \subset K(\sqrt[q]{a},\beta)$.    Further, $KF=K(\sqrt[q]{a},\beta)/K$ is a cyclic extension of degree $q^r$ for some $r>0$, no factor of $q$ ramifies in this extension  (by Proposition 8 of Chapter II, \S4 of \cite{L}), and no factor of $q$ splits in the extension $K(\sqrt[q]{a},\beta)/K(\beta)$ by Lemma \ref{le:notsplit}.  By Lemma \ref{makereal} we can also assume $a \in \Omega_q(K(b))$.  

Since factors of $q$ in $N(\beta)$ do not ramify in the extension $N(\beta,\sqrt[q]{a})/N(\beta)$, if for some factor $\qq_{N(\beta)}$ of $q$ in $N(\beta)$ we have that $\ord_{\qq_{N(\beta)}}a \not =0$, we also must have $\ord_{\qq_{N(\beta)}}a \equiv 0 \mod q$.  Thus without loss of generality (multiplying $a$ by $q$-th powers of some elements of $N(\beta)$, if necessary), we can assume that $a$ has no occurrences of factors of $q$ in its divisor.  Note that if $q=2$, we would only be multiplying $a$ by squares and thus not changing the fact that $a \in \Omega_2(N(\beta))$.

Now let $\calA_{N(\beta)} \subseteq \calC_{N(\beta)}(q)=\calQ_{N(\beta)}$ and let $d \in N(\beta)$ be such that for all primes $\qq_{N(\beta)} \in \calA_{N(\beta)}$ we have that  $\ord_{\qq_{N(\beta)}}d \not \equiv 0 \mod q$, $\ord_{\qq_{N(\beta)}}d \leq -3\ord_{\qq_{N(\beta)}}q$ and $d$ has no other poles.  As above, such a $d$ exists by the Strong Approximation Theorem.

 The reason for possibly choosing a subset of  factors of $q$ is to point out that in principle we don't have to treat all the factors of $q$ the same way, i.e.  we may want to allow some of the factors in ``denominators'', while banning others.    The proposition below lets us bound the order of the poles the elements of our field can have at factors of $q$ in $\calA_{N(\beta)}$, while imposing no constraints on other factors of $q$.

\begin{proposition}
\label{onlyq}
Let $E_{\inff}$ be defined as above, let  
\[
F_{\inff}=E_{\inff}(\sqrt[q]{1+ d^{-1}},\sqrt[q]{1+ (dx^p+d^p)^{-1}},\sqrt[q]{1+(a+a^{-1})d^{-1}}),
\]
 and let
\[
C(E_{\inff},a,d,q)=\{x \in K_{\inff}|\exists y \in F_{\inff}(\sqrt[q]{a}): {\mathbf N}_{F_{\inff}(\sqrt[q]{a})/F_{\inff}}(y) = dx^q+d^q\}.
\]
We claim $C(E_{\inff},a,d,q)=\{x \in K_{\inff}| \forall K \in I_{N(\beta,x)}\, \forall \qq_K \in \calA_K : \ord_{\qq_K}x > \frac{q-1}{q}\ord_{\qq_K}d \}$.   
\end{proposition}
\begin{proof}
The proof of this proposition is almost identical to the proof of Proposition \ref{finitely many for p not q} except that it relies on Proposition \ref{prop:badprimeq} and Proposition \ref{prop:fixorderq}  in lieu of Proposition \ref{prop:badprime} and Proposition \ref{prop:fixorder}.
\end{proof}

As above, Lemma \ref{integral} allows us to use the definition of $C(E_{\inff},a,d,q)$ to obtain a definition $R_{K_{\inff}, \calA_{\inff}}$.  With the definition of  $R_{K_{\inff}, \calA_{\inff}}$ in mind, we now modify slightly the definition in Corollary \ref{def1} to replace $ \Phi_q(K_{\inff},\calS_{K_{\inff}}) \cap \Phi_q(K_{\inff})$ with  an expression involving  $R_{K_{\inff},\calW_{\inff}}$.   We also state the corresponding definitions of $O_{K_{\inff},\calS_{K_{\inff}}}$ for the case where $\calS_G \cap \calQ_G = \emptyset$,  and $O_{K_{\inff}}$. Let $w, \hat w \in G$ be such that 
\be
\item $\ord_{\qq_G}w=3\ord_{\qq_G}q$ for any $\qq_G \in \calC_G(q)$, 
\item $\ord_{\pp_G}w =1$ for any $\pp_G \in \calS_G$,
\item  $w$ has no other zeros. 
\item  $\ord_{\qq_G}\hat w=3\ord_{\qq_G}q$ for any $\qq_G \in \calC_G(q)$, 
\item $\hat w$ has no other zeros. 
\ee
(As above such elements $w$ and $\hat w$ exist by the Strong Approximation Theorem.)
\begin{corollary}
\label{diffversion} 
\be
\item
 \[
 x \in O_{K_{\inff},\calW_{K_{\inff}}}, x \not=0
 \]
 \[
 \Updownarrow
 \]
 \[
\forall c \mbox{ such that } \left(\frac{(c-1)}{w} \in R_{K_{\inff},\calW_{\inff}} \land c \in \Omega_q(K_{\inff}) \right)\forall b \in K_{\inff} 
\]
\[
 ((bx^q+b^q=0) \lor  \exists y \in L_{\inff}(\sqrt[q]{c}):{\mathbf N}_{L_{\inff}(\sqrt[q]{c})/L_{\inff}}(y) = bx^q+b^q).
\]
\item   \[
 x \in O_{K_{\inff},\calS_{K_{\inff}}}, x \not =0 
 \]
 \[
 \Updownarrow
 \]
 \[
x \in R_{K_{\inf}, \calQ_{\inff}} \land \forall c \mbox{ such that } \left(\frac{(c-1)}{w} \in R_{K_{\inff},\calW_{\inff}} \land c \in \Omega_q(K_{\inff}) \right)\forall b \in K_{\inff}
\]
\[
 ((bx^q+b^q=0) \lor \exists y \in L_{\inff}(\sqrt[q]{c}): {\mathbf N}_{L_{\inff}(\sqrt[q]{c})/L_{\inff}}(y) = bx^q+b^q).
\]
\item  \[
 x \in O_{K_{\inff}}, x \not=0 
 \]
 \[
 \Updownarrow
 \]
  \[
x \in R_{K_{\inf}, \calQ_{\inff}} \land \forall c \mbox{ such that } \left(\frac{(c-1)}{\hat w} \in R_{K_{\inff},\calQ_{\inff}} \land c \in \Omega_q(K_{\inff}) \right)\forall b \in K_{\inff} 
\]
\[
((bx^q+b^q=0)\lor \exists y \in L_{\inff}(\sqrt[q]{c}): {\mathbf N}_{L_{\inff}(\sqrt[q]{c})/L_{\inff}}(y) = bx^q+b^q).
\]
\ee
\end{corollary}
\begin{proof}
We show that the first formula defines the right set.  The argument for the other two definitions is similar.  It is enough to observe the following.  In any $K \in I_N$ the numerator of the divisor of $c-1$ is divisible by the numerator of the divisor of $q^3$ and by every $\pp_K$ in $\calS_K$.    Thus  $c \in \Theta_q(K,\calS_K) \cap \Phi_q(K)$.  Conversely, if $c \in \Theta_q(K,\calS_K) \cap \Phi_q(K)$, then  the divisor $c-1$ is divisible by the numerator of the divisor of $q^3$ and by every $\pp_K$ in $\calS_K$ and therefore $\frac{c-1}{w}$ does not have any poles at primes of $\calW_K$, so that $\frac{(c-1)}{w} \in R_{K_{\inff},\calW_{\inff}}$.
\end{proof}

\begin{theorem}
\label{thm:maindef}
Let $p, q$ be rational prime numbers, not necessarily distinct.  Let $H$ be a number field, and let $H_{\inff}$ be an algebraic extension of $H$.  Let $\calS_H$ be a finite, possibly empty, set of primes of $H$.  Assume all primes of $H$ not in $\calS_H$ are heredeterily $q$-bounded in $H_{\inff}$, and primes in $\calS_H$ and factors of $q$ are completely $p$-bounded in $H_{\inff}$.  In this case, the integral closure of $O_{H,\calS_H}$ in $H_{\inff}$ is first-order definable over $H_{\inff}$.
\end{theorem}
\begin{proof}
Given an arbitrary number field $H$ and an algebraic extension $H_{\inff}$ of $H$, not necessarily containing any roots of unity required above, we have to show that the norm equations we have been using in our definitions can be rewritten as polynomial equations with relevant solutions in $H_{\inff}$. Below we present an informal outline of this rewriting process.  For a more general and formal discussion of the rewriting techniques we refer the reader to the section on coordinate polynomials in \cite{Sh34}.  Let $G=H(\xi_q, \xi_p), K_{\inff}=H_{\inff}(\xi_q,\xi_p)$. 

  We start with rewriting the norm equation itself.   If $T$ is any field of characteristic 0 and $c \in T \setminus T^q$, $u_1,\ldots, u_{q}, z \in T$, $y = \sum_{i=1}^{q}a_i\sqrt[q]{c}^{(i-1)}$, then
\begin{equation}
\label{eq:normrewrite}
{\mathbf N}_{T(\sqrt[q]{c})/T}(y)-z =\prod_{j=0}^{q-1}\sum_{i=1}^{q}u_i\xi_q^{(i-1)j}\sqrt[q]{c}^{(i-1)}-z=N(u_1,\ldots,u_{q},c,z) \in \Z[U_1,\ldots,U_{q}, C, Z],
\end{equation}
and the coefficients of $N(U_1,\ldots,U_{q}, C, Z)$ depend on $q$ only.

If $c, w \in T, c=w^q$, then for any $z \in T$ the equation $N(U_1,\ldots,U_q,c,z)=0$ has solutions $a_1,\ldots, a_{q} \in T(\xi_q)$.
Indeed, consider the following system of equations:
\[
\left \{
\begin{array}{c}
\sum_{i=0}^{q-1}a_iw^{i}=z,\\
\sum_{i=0}^{q-1}a_i\xi_q^{ij}w^{i}=1, j=1,\ldots,q-1
\end{array}
\right .
\]
This is a nonsingular system with a matrix $(\xi_q^{ij}w^{i}), i=0,\ldots, q-1,j = 0,\ldots,q-1$ having all of its entries in $T(\xi_q)$.  Since the vector $(z,1,\ldots,1)$ also has all of its entries in $F(\xi_q)$,we conclude that the system has a unique solution in $F(\xi_q)$.
So if we, for example, consider ${\mathbf N}_{L_{\inff}(\sqrt[q]{c})/L_{\inff}}(y) = bx^q+b^q$ with  potential solutions $y$ ranging over $L_{\inff}(\sqrt[q]{c})$, then we can conclude that that this norm equation is equivalent to a polynomial equation 
\begin{equation}
\label{stepdown}
N(u_1,\ldots,u_{q},c,bx^q+b^q) =0
\end{equation}
 with coefficients in $\Z$ and potential solutions 
 \[
 u_1,\ldots, u_q \in L_{\inff}=K_{\inff}(\sqrt[q]{1+ x^{-1}}, \sqrt[q]{1+ (bx^q+b^q)^{-1}},\sqrt[q]{ 1+(c+c^{-1})x^{-1}}).  
\]
We now would like to replace \eqref{stepdown} by an equivalent equation but with solutions in 
\[
L_{2,\inff}=K_{\inff}(\sqrt[q]{1+ x^{-1}}, \sqrt[q]{1+ (bx^q+b^q)^{-1}}).
\]  We have to consider two options: either there exists $\gamma \in L_{2,{\inff}}$ such that 
\begin{equation}
\label{eq:gamma}
\gamma^q=1+(c+c^{-1})x^{-1}
\end{equation}
 and in this case all the solutions $u_1,\ldots, u_q \in L_{2, \inff}$, or $1+(c+c^{-1})x^{-1}$ is not a $q$-th power in $L_{2,\inff}$ so that $u_i=\sum_{j=0}^{q-1}u_{i,j}\gamma^j$, where $\gamma$ is as in \eqref{eq:gamma} and $u_{i,j} \in L_{2,\inff}$.  In the latter case we can rewrite \eqref{stepdown} first as 
\begin{equation}
\label{eq:system}
N(\sum_{j=0}^{q-1}u_{1,j}\gamma^j,\ldots,\sum_{j=0}^{q-1}u_{q,j}\gamma^j,c,bx^q+b^q) =0,
\end{equation}
 and then as a system of equations over $L_{2,\inff}$ using the fact that the first $q-1$ powers of $\gamma$ are linearly independent over $L_{2,\inff}$.  In other words, we rewrite \eqref{eq:system} first as 
\begin{equation}
\label{eq:powers}
\sum_{i=0}^{q-1}N_i(u_{1,0},\ldots, u_{q,q-1},c,b, x)\gamma^i =0,
\end{equation}
where $N_i$ are polynomials in listed variables with coefficients in $\Z$, by systematically replacing $\gamma^q$ via $1+(c+c^{-1})x^{-1}$ and clearing the denominators (i.e. clearing $c$ from denominators by multiplying through by a sufficient high power of $c$), and then as a system 
\begin{equation}
\label{eq:ind}
\bigwedge_{i=0}^{q-1} N_i(u_{1,0},\ldots, u_{q,q-1},c,b, x)=0.
\end{equation}
Note that, even if $\gamma \in L_{2,\inff}$, we can still replace \eqref{stepdown} by \eqref{eq:ind}.  To see this reconsider \eqref{eq:system} as
\begin{equation}
\label{eq:systemv}
N(\sum_{j=0}^{q-1}U_{1,j}\Gamma^j,\ldots,\sum_{j=0}^{q-1}U_{q,j}\Gamma^j,C,BX^q+B^q) =0,
\end{equation}
with $U_{i,j}, X, C, B, \Gamma$ algebraically independent over $\tilde \Q$,  and produce a system of equations 
\begin{equation}
\label{eq:ind2}
\bigwedge_{i=0}^{q-1} N_i(U_{1,0},\ldots, U_{q,q-1},C,B, X)=0
\end{equation}
by the process described above, first systematically replacing  $\Gamma^q$  by $1+(C+C^{-1})X^{-1}$ and clearing the denominators (this time removing $C$ from denominators), and then treating the first $q-1$ powers of $\Gamma$ as linearly independent over the polynomial ring $\tilde \Q[U_{i,j},X, X^{-1}, C, C^{-1}, B]$. Observe that the left side of \eqref{eq:systemv} is equivalent to 
\[
\sum_{i=0}^{q-1} N_i(U_{1,0},\ldots, U_{q,q-1},C,B, X)\Gamma^i
\]
 modulo the ideal $(\Gamma^q -1-(C+C^{-1})X^{-1})$ in the ring $\tilde \Q[X, X^{-1}, C, C^{-1}, U_{i,j} B, \Gamma]$.  In other words for any values of $U_{i,j}, X \not =0, C\not =0, B$ in $\tilde \Q$   satisfying the system \eqref{eq:ind2}, we have that \eqref{eq:systemv} will be satisfied as long as $\Gamma$ is set to a value $\gamma \in \tilde \Q$ satisfying \eqref{eq:gamma} where $c$ and $x$ are the $\tilde \Q$-values assigned to $C$ and $X$ respectively. Of course if \eqref{eq:system} has solutions in $L_{2,\inff}$ we can always find solutions to the system \eqref{eq:ind} whether or not $\gamma \in L_{2,\inff}$.

Thus for any $c, b, x \in K_{\inff}$ we can conclude that \eqref{stepdown} has solutions $u_1,\ldots, u_q \in L_{\inff}$ if and only if there exist $u_{1,0}, \ldots, u_{q, q-1} \in  L_{2,\inff}$ satisfying \eqref{eq:system}.

  Proceeding in the same fashion we can eventually obtain an equivalent system of equations with potential solutions in $K_{\inff}$.  Now if a given field $H_{\inff}$ does not contain $\xi_q$ or $\xi_p$, then we can rewrite all the equations one more time so that the final system has solutions   and coefficients in $H_{\inff}$.  
\end{proof}
We can also separate out  results concerning integrality at finitely many primes.
  \begin{theorem}
  \label{finite}
The following statements are true.
\be
\item  If a $G$-prime $\pp_G$ is completely $q$-bounded, $M$ is a $q$-bounding field for $\pp_G$,  $b \in K_{\inff}$ is such that for some $\pp_{M(b)} \in \calC_{M(b)}(\pp_G)$ we have that $\ord_{\pp_{M(b)}}b \not \equiv 0 \mod q \land \ord_{\pp_{M(b)}}b <0$, and $b$ has no other poles, then the set of all elements $x \in K_{\inff}$ such that $\ord_{\pp_{M(x,b)}}x \geq \frac{q-1}{q} \ord_{\pp_{M(x,b)}}b$ for all $\pp_{M(x,b)} \in \calC_{M(b,x)}(\pp_{M(b)})$ is existentially definable.   (For future reference in Section \ref{elliptic curves} denote this set by $\mbox{Int}(b,\pp_{M(b)}, q)$.)
\item  If  ramification degrees over $G$ of all factors of $\pp_G$  in  number fields contained in $I_G$ are uniformly  bounded, then the integral closure of the valuation ring of $\pp_G$ in $K_{\inff}$ is existentially definable.  
\ee
\end{theorem}
We now make use of unbounded primes.
\begin{theorem}
Let $\calS_G\cup \{\mbox{ factors of } q\}$ be a completely $q$-bounded in $K_{\inff}$ finite set of primes of $G$, and let $R_{\inff,\calS_G}$ be a subring of $K_{\inff}$ such that $x \in R_{\inff,\calS_G}$ if and only if in $G(x)$ the poles of  $x$ are either factors of $q$ or primes of $\calS_G$, or are at primes that are $q$-unbounded.  In this case $R_{\inff, \calS_G}$ is first-order definable over $K_{\inff}$.
\end{theorem}
\begin{proof}
It is enough to consider what happens to the solvability of the norm equation below for $c$ chosen so that factors of $q$ and primes in $\calS_G$ split and $x$ has poles only at the primes described in the statement of the theorem. So let $K \in I_G$ and consider
\begin{equation}
\label{old}
 {\mathbf N}_{L_K(\sqrt[q]{c})/L_K}(y) = bx^q+b^q.
\end{equation}
 As above, since factors of $q$ and primes in $\calS_G$ split, this equation will be solvable locally at these primes.  Now as far as unbounded primes are concerned, we can always consider the norm equation over a field $K$ large enough so that  factors of the unbounded primes occurring with a non-zero order in the divisor of the right-side of \eqref{old}   either ramify with ramification degree divisible by $q$ or their relative degree goes up by a factor divisible by $q$.  Over this $K$, either these factors split completely when we adjoin the $q$-th root of $c$ or the right side of \eqref{old} has order divisible by $q$ at the factors of these $q$-unbounded primes.  Thus, in any case of large enough $K$, the norm equation is solvable at all the factors of unbounded primes.
\end{proof}
One can prove a few more variations of such results.  The theorem below is another example.  Its proof is completely analogous to the proofs above.
\begin{theorem}
\label{general}
Let $P=\{p_1,\ldots,p_k\}$ be a finite set of rational primes such that each prime of $G$ not dividing any element of $P$,  is heredeterily $p_i$-bounded in $K_{\inff}$ with respect to some $p_i$, and each $p_i$ is completely $p_j$-bounded in $K_{\inff}$ for some $p_j$.  In this case $O_{K_{\inff}}$ is first-order definable over $K_{\inff}$.
\end{theorem}

\section{Examples of Infinite Extensions of $\Q$ where the Ring of Integers is First-Order Definable}
\label{examples}
In this section we describe some example to which our methods apply.  Some of these examples will be pretty straightforward while others are more esoteric.  We start with the more straightforward examples.
\begin{example}[Fields with Uniformly Bounded Local Degrees]
Perhaps the simplest example of a $q$-bounded infinite extension of rationals is an infinite extension where the local degrees of all primes are uniformly bounded.  In such a field every prime is completely $q$-bounded for any prime $q$.  An example of such an extension is an infinite Galois extension generated by all extensions of degree $p$ (for a fixed prime $p$) of $\Q$ contained in cyclotomics. More examples of such fields can be found in \cite{Checcoli}.    Most of such examples where the field is Galois over $\Q$ were already covered by definability results of Videla with respect to the ring of integers.  However, one can construct many non-Galois examples of such fields.   It is enough to take a collection $\{K_i\}$ of number fields which are Galois but not abelian over $\Q$, linearly disjoint over $\Q$, of degree less or equal to some fixed $n$ over $\Q$, and consider a collection of number fields $\{N_i\}$, where $N_i \subset K_i$ and $N_i$ is not Galois over $\Q$.  Now let $N_{\inff}$ be the compositum of all $N_i$ inside $\tilde \Q$.  If $K_{\inff}$ is the compositum of all $K_i$ inside $\tilde \Q$, then $N_{\inff} \subset K_{\inff}$ and $[K_{\inff}:N_{\inff}]=\infty$.   Thus, while Videla's results give us a first-order definition of $O_{K_{\inff}}$ over $K_{\inff}$, they do not give us a first-order definition of $O_{N_{\inff}}$ over $N_{\inff}$, obtainable by our methods.  
\end{example}

\begin{example}[Galois extensions without cyclic  subextensions of degree divisible by arbitrarily high powers of  $q$]
\label{Galois}
If $K_{\inff}$ is a Galois extension of a number field $G$ such that for any Galois $K \in I_G$, we have that $[K:G] \not \equiv 0 \mod q$, then $O_{K_{\inff}}$ and the integral closure of any ring of $\calS$-integers in $K_{\inff}$ is first-order definable over $K_{\inff}$. 

  It is not hard to see that in this case ramification and relative degrees in all finite subextensions are prime to $q$ and thus all the primes are completely $q$-bounded.  This example  covers cyclotomic extensions with finitely many ramified primes, i.e. extensions of the form $\Q(\xi_{p^{\ell}_1}, \ldots, \xi_{p^{\ell}_k}, \ell \in \Z_{>0})$, where $p_1,\ldots,p_k$ are rational primes, and all their subfields that include all abelian extensions with finitely many ramified primes.  (The definability of rings of integers in these extensions follows from Videla's results.)

 Given a prime $q$, and an integer $m>0$,  our method also applies to the case of  a cyclotomic extension (and any of it subfields) generated by the set  
\[
\{\xi_{p^{\ell}}| \ell \in \Z_{>0}, p \not=q \mbox{ is any prime such that } q^{m +1}\not | (p-1)\}.
\]
  (In other words we need to omit primes occurring in the arithmetic sequence $kq^{m+1} +1, k \in \Z_{>0}$, and by increasing $m$, we can make the density of the omitted primes arbitrary small.)   This example generalizes an example of Fukuzaki where he defined integers over the field $\Q(\{\cos(2\pi/\ell^n ) : \ell\in \Delta, n \in \Z_{>0}\})$  and any of its Galois subextensions, and where $\Delta$ is the set of all the prime integers which are congruent to −1 modulo 4.

  On top of such a  cyclotomic field we can also add a field generated by any subset of $p$-th roots of algebraic numbers contained in this cyclotomic field, with $p$ as above not equivalent to 1 modulo $q^{m+1}$.  Clearly, many more examples of Galois extensions of this sort can be generated.
   
\end{example}
As we pointed above, being Galois is not required for our method to work.  Thus we have some obvious examples of non-Galois extensions where we can define integers.
\begin{example}[Extensions that are not necessarily Galois]
If $K_{\inff}$ is a tower of finite extensions of degree less than some positive integer $m$,   then $O_{K_{\inff}}$ and the integral closure of any ring of $\calS$-integers in $K_{\inff}$ are first-order definable over $K_{\inff}$.  Observe that a field of this sort can have primes of arbitrarily large or infinite local degree, and thus this example is a non-trivial generalization of the first example.

If the extension is Galois, we are looking at a field discussed in the second example.  So the new cases will come from extensions that are not Galois.  Observe, that in such a field for any $q >m$ all the primes are completely $q$-bounded.
\end{example}
It is more difficult to describe examples where primes are not necessarily completely $q$-bounded.
\begin{example}[Less natural fields]
Let $q$ be a rational prime and let $\{p_1,\ldots\}$ be a listing of all rational primes omitting $q$.  Let $\pi_{i}=\prod_{j=1}^ip_j$.  Let $G$ be any number field and let $\{\pp_1,\ldots\}$ be a listing of all primes of $G$ not lying above $q$.  We construct a tower of fields starting with $G$ where  all factors of  $q$ are completely $q$-bounded, all the other primes of $G$ and any finite extension of $G$ are $q$-bounded but not completely $q$- bounded and are $p$-unbounded for any other prime $p$. Let $K_0=G$ and assume we have constructed $K_1,\ldots,K_n$ for some $n\geq 0$.  We now construct $K_{n+1}$ in three steps.  

First we construct an extension $M_{n,1}$ of $K_n$ of degree  $\pi_n$, where all the primes above $\pp_1,\ldots,\pp_n$ will have ramification degrees divisible by $\pi_n$ and all the primes above $q$ split completely. (Such an extension always exist.  For example take an element $a$ of $O_{K_n}$ such that $\ord_{\pp_i}a=1$ for $i =1,\ldots,n$ and $a \equiv 1$ mod $q$ and adjoin $\sqrt[\pi_i]{a}$ to $K_n$.)  This step insures that the ramification degree of factors of any prime of $G$ not dividing $q$ will eventually be divisible by arbitrarily high powers of rational primes distinct from $q$. 

We now construct a non-trivial extension $M_{n,2}$ of $M_{n,1}$ where all the factors of $\pp_1, \ldots, \pp_n$ and $q$ in $M_{n,1}$ split completely into distinct factors.    (For example we can adjoin $\sqrt[p]{b}$, where $p$ is prime to $\pp_1,\ldots, \pp_n$ and $q$ and $b \equiv 1 \mod (q\pp_1\ldots \pp_n)$.)  This step allows us to produce $q$-bounded and $q$-unbounded paths above every prime.

Finally $K_{n+1}$ is an extension of  $M_{n,2}$ of degree $q$ satisfying the following requirements:
\be
\item  All the factors of $q$ split completely.
\item For each $i=1,\ldots,n$ and each $\ttt_i$ that is a factor of some $\pp_i$ in $M_{n,1}$, if  $\ttt_{i,1},\ldots,\ttt_{i,k}$ are factors of $\ttt_i$ in $M_{n,2}$ under some ordering, then $\ttt_{i,1}$ splits completely into distinct factors and $\ttt_{i,2},\ldots, \ttt_{i,k}$ do not split in the extension $K_{n+1}/M_{n,2}$.
\ee
To construct such an extension, by Lemma \ref{notq}, we can take a $q$-th root of an algebraic integer of $M_{n,2}$ such that it is equivalent to $1 \mod q^{3}$ and modulo $\ttt_{i,1}$, and to a non-$q$-th power modulo $\ttt_{i,j}, j\geq 2$.  In this step we construct the next level of $q$-bounded and $q$-unbounded paths.  At the ``end'' of the construction every prime of any $K_{n}$ not dividing $q$,  will lie along the ``left-most'' $q$-bounded path, and the ``right-most'' $q$-unbounded path. (In fact, every prime not dividing $q$ will lie along infinitely many $q$-bounded and $q$-unbounded paths.)

We now let $K_{\inff} = \bigcup_{i=1}^{\infty} K_i$. 

 It is easy to see that for all $K \in I_G$ every factor of $q$ is unramified and of relative degree 1.  At the same time, for any $p\not =q$, any positive integer $m$,  and any $\pp_i$ prime to $q$, there is a field $K \in I_G$ where all the factors of $\pp_i$ have a ramification degree over $\pp_i$ divisible by $p^m$.  
 
Further, for $i \in \Z_{>0}$, let $d_{i}=\max_{\pp_{K_{i+1}} \in \calC_{K_{i+1}}(\pp_i)}\{\ord_q(d(\pp_{K_{i+1}}/\pp_i))\}$, and note that for any $\pp_i$, for any $K \in I_{G} $ there exists a $K$-factor $\pp_K$ of $\pp_{i}$ such that $\ord_q(d(\pp_K/\pp_i))\leq d_{i}$, while at the same time for any  $m \in \Z_{>0}$,  there exist a field $M\in I_G$ and an $M$-factor $\pp_M$ of $\pp_i$ such that $f(\pp_M/\pp_i) \equiv 0 \mod q^m$.

\end{example}

We can also produce an example where one would need Theorem \ref{general}.  The construction is similar to the one above and, in particular, the existence of extensions we need can be justified by similar arguments.
\begin{example}[Also not very natural fields]
Let $Q=\{q_1,\ldots,q_m\}$ be a finite collection of rational primes.  Let $\{p_1, \ldots \}$ be a listing of all rational primes excluding the primes in $Q$.  Let $\pi_i=\prod_{j=1}^ip_j$.  Let $G$ be any number field and divide all the primes of $G$ not lying above any prime of $Q$ into $m$ classes with $\{\pp_{i,j}, i=1,\ldots,m, j \in \Z_{>0}\}$. We now construct a tower of fields $\{K_i\}$ with $K_{\inff}$, as above, being the union of the tower. Let $K_0=G$ and assume that $K_{n}$ for some $n \geq 0$ has been constructed.  We construct $K_{n+1}$ in $m+1$ steps.  First let $M_{0,n}/K_n$ be an extension of degree $\pi_{n+1}$ such that 
\be
\item All the primes above primes of $Q$ split completely.
\item All the primes in the set $\{\pp_{i,j}, i =1,\ldots, m, j = 1,\ldots, n+1\}$ ramify completely.
\ee
Next we construct construct $M_{i,n}/M_{i-1,n}$ for $i=1,\ldots,m$.  First of all, the degree of the extension will be $q_i$.  Secondly, all the primes above the primes of $Q$ and all the primes above the primes in the set $\{\pp_{i,j}, j=1,\ldots, n+1\}$  split completely.  Thirdly, all the primes in the set $\{\pp_{r,j}, r=1,\ldots, m, r \not = i, j=1,\ldots, n+1\}$ remain prime.  Finally, $K_{n+1}=M_{m,n}$.  

It is not hard to see that for each $i=1,\ldots, m$ the primes $\{\pp_{i,j}, j=1,\ldots, \}$ of $G$  are completely $q_i$-bounded and these primes are $p$-unbounded for any prime $p \not=q_i$.  Further, all the primes above primes of $Q$ are completely $q$-bounded for any prime $q$.  Thus we need to use Theorem \ref{general} here to get the desired definitions.

\end{example}

We should finish this section with a listing of some obvious fields which are {\bf not} $q$-bounded: the algebraic closure of $\Q$, the maximal abelian extension of $\Q$, the field of all totally real numbers, the field of real algebraic numbers. (Add here the definable result over the field.) In general examples of such fields are also not hard to generate. We remind the reader that one would expect the field of all totally real numbers not to be $q$-bounded since, as has been noted above, the first-order theory of the field of all totally real integers is decidable, while this is not the case for the ring of integers of this field.  Thus, the ring of integers of the field of all totally real integers does not have a first-order definition over it fraction field.

\section{From Undecidability of Rings to Undecidability of Fields}
\label{undecidability}
We start with reviewing results we are going to use due to L. Kronecker, J. Robinson and the author of this paper.  We start with reviewing the results of Julia Robinson from \cite{Rob3}.
\begin{theorem}[JR]
\label{Julia}
 The natural numbers can be defined arithmetically in any totally
real algebraic integer ring $R$ such that there is a smallest interval $(0, s)$, $s$
real or $\infty$, which contains infinitely many sets of conjugates of numbers of
$R$, i.e., infinitely many $x \in R$ with all the conjugates (over $\Q$) in $(0, s)$.
\end{theorem}
J. Robinson showed in \cite{Rob3} that certain infinite towers of totally really quadratic extensions have rings of integers with $s=\infty$ and thus the first-order theory of these rings is undecidable.  C. Videla used this result in \cite{V1} to show that the archimedean hull of $\Q$ is undecidable.  Further, J. Robinson (\cite{Rob3}), C. Videla (\cite{V2}), and K. Fukuzaki  (\cite{Fuk}) make use of the following proposition which is a consequence of a result by L. Kronecker from \cite{Kron}.
\begin{proposition}[Kronecker]
\label{Kron}
The interval $(0, 4)$
contains infinitely many sets of conjugates of totally real algebraic integers
and  no sub-interval of $(0, 4)$ does.
\end{proposition}

An immediate consequence of Theorems \ref{Julia} and \ref{Kron} is that any ring of totally real integers containing a set of the form $\{\cos\frac{2\pi}{m}, m  \in \calZ \}$ with $\calZ$ infinite, one can give a first-order definition of integers.  Thus, extending results of K. Fukuzaki, we now have the following theorems.
\begin{theorem}
\label{undecidable}
Let $q$ be a rational prime, let $m>0$ be an integer and let 
\[
K_{\inff}=\Q(\cos(2\pi/n),  n =\prod_{i=1}^s p_i^{\ell_i},  p_i \not \equiv 1 \mod q^m, s, \ell_1,\ldots, \ell_s \in \Z_{>0}),
\]
 where $p_i$ range over all primes satisfying the condition $p \not \equiv 1 \mod q^m$.  In this case the first-order theory of $K_{\inff}$ is undecidable and $\Z$ is first-order definable $K_{\inff}$.
\end{theorem}

Since the ring of all totally real integers is undecidable, every ring of totally real integers is trivially contained in an undecidable ring.  However, this is not automatically clear for the fields, since the first-order theory of the field of all totally real numbers is decidable.   While we cannot show that the first-order theory of any $q$-bounded totally real field is undecidable, we can show the following.
\begin{theorem}
Any $q$-bounded totally real field is contained in a totally real field that has a first-order definition of rational integers and thus has an undecidable first-order theory. 
\end{theorem}
\begin{proof}
Let $K_{\inff}$ be a $q$-bounded field and observe that $K_{\inff}(\cos(2\pi/p^k), k \in \Z_{>0})$ for some $p \not = q$ is also $q$-bounded, since we will introduce at most a finite number of subextensions of degree divisible by $q$.  (In other words, the increase in divisibility by $q$ of relative or ramification degrees can come only from adding the extension $\Q(\cos 2 \pi/p)$ of degree $(p-1)/2$ over $\Q$.) But  the ring of integers of the extended field is now undecidable and has a first-order definition of the rational integers, by the discussion above.  Thus, since the extended field is still $q$-bounded, we have that the extended field has a first-order definition of rational integers and an undecidable first-order theory.

\end{proof}
We now turn our attention to non-real fields.  In \cite{V2}, C. Videla showed that the ring of integers is definable in infinite Galois extensions of $\Q$ where the degree of every finite subextension is a product of a fixed finite set of primes.   Further, as mentioned above, in \cite{V3}, Videla proved using  a theorem 
of J. Robinson that the ring of integers of $\Q(\xi_{p^r}, r \in \Z_{>0})$ is undecidable.  Combining the two results, he also obtained the first-order undecidability of $\Q(\xi_{p^r}, r \in \Z_{>0})$.    

Below we prove the following theorem.
\begin{theorem}
\label{cyclotomics}
Rational integers are first-order definable in any abelian extension of $\Q$ with finitely many ramified primes, and therefore the first-order theory of such fields is undecidable.
\end{theorem}
  Rather than relying on the result of J. Robinson, we use existential definability and undecidability results from \cite{Sh36} and \cite{Sh17}, where the following result was proven.
\begin{theorem}
\label{exist}
 Let $A_{\inff}$ be an abelian (possibly infinite) extension of $\Q$ with finitely many ramified
primes. Then for any number field $A \subseteq A_{\inff}$ and any finite non-empty set $\calS_A$ of its primes, we have that
$\Z$ is existentially definable in the integral closure of $O_{A,\calS_A}$ in $A_{\inff}$.
\end{theorem}
Now Theorem \ref{cyclotomics} follows from the fact that any abelian extension with finitely many ramified primes must by L. Kronecker's Theorem be a subfield of a cyclotomic extension with finitely many ramified primes, i.e. an extension where prime divisors of the degrees of all finite subextensions come from a  finite set of primes.  Such an extension is $q$-bounded for any odd $q$ not occurring in the above mentioned finite set of primes, by Example \ref{Galois}.  Further, all the primes of $\Q$ are completely $q$-bounded for such a $q$.  Thus, any small subring is first order definable over any abelian extension of $\Q$ with finitely many ramified primes, and therefore by Theorem \ref{exist} we conclude that rational integers are first-order definable over any abelian extension of $\Q$ with finitely many ramified primes.  Since the set of non-zero integers is definable over any ring of algebraic integers, we can ``simulate'' the field over the ring of integers, and therefore obtain the following corollary:
\begin{corollary}
\label{cor:int}
Rational integers are first-order definable in the  ring of integers of any abelian extension of $\Q$ with finitely many ramified primes, and therefore the first-order theory of such a ring is undecidable.
\end{corollary}  
 \section{Using Elliptic Curves with Finitely Generated Groups}
 \label{elliptic curves}
 \setcounter{equation}{0}
 In this section we show that over the fields with finitely generated elliptic curves, assuming there exists at least one completely $q$-bounded prime, we can define $\Z$ and conclude that the first-order theory is undecidable.  
 
 The use of elliptic curves to investigate definability and decidability has a long history.  Perhaps the first mention of elliptic curves in the context of the first-order definability belongs to R. Robinson in \cite{RRobinson} and in the context of existential definability  to J. Denef in \cite{Den3}.  Using elliptic curves B. Poonen has shown in \cite{Po} that if  for a number field extension $M/K$ we have an elliptic curve $E$ defined over $K$, of rank one over $K$, such that
the rank of $E$ over $M$ is also one,  then $O_K$ (the ring of integers of $K$) is Diophantine over $O_M$.   G. Cornelissen, T. Pheidas and K. Zahidi weakened somewhat assumptions of B. Poonen's theorem. Instead of requiring a rank 1
curve retaining its rank in the extension, they require existence of a rank 1 elliptic curve over the bigger field
and an abelian variety over the smaller field retaining its positive rank in the extension (see \cite{CPZ}).  Further,
B. Poonen and the author have independently shown that the conditions of B. Poonen's theorem can be weakened to
remove the assumption that the rank is one and require only that the rank in the extension is positive and the same as the rank over the ground field
(see \cite{Sh33} and \cite{Po3}).   In \cite{CS} G. Cornelissen and the author of this paper used elliptic curves to define a subfield of a number field using one universal and existential quantifiers.  

Elliptic curves specifically of rank 1 have been used in several papers in connection to discussions of definability and decidability over big subrings of number fields (i.e. subrings where infinitely many, though not all, primes are inverted).  See \cite{Po2}, \cite{PS}, \cite{EE}, \cite{Perlega}, \cite{EES} and \cite{Sh38}.

  Following J. Denef in \cite{Den2}, as has been mentioned above, the author also considered the situations where elliptic curves had finite rank in
infinite extensions and showed that when this happens in a totally real field one can existentially define $\Z$ over the ring of integers of this field
and the ring of integers of any extension of degree 2 of such a field (see \cite{Sh37}).

Recently, in \cite{MR}, B. Mazur and K. Rubin showed that if Shafarevich-Tate conjecture held over a number field $K$, then for any prime degree cyclic extension $M$ of $K$, there existed an elliptic curve of rank one over $K$, keeping its rank over $M$.  Combined with B. Poonen's theorem, this new result shows that Shafarevich -Tate conjecture implied HTP is undecidable over the rings of integers of any number field.  

C. Videla also used finitely generated elliptic curves to produce undecidability results.  His approach, as discussed above, was based on an elaboration by C. W. Henson of a proposition of J. Robinson and  results of D. Rohrlich (see \cite{Rohrlich}) concerning finitely generated elliptic curves in infinite algebraic extensions. 

The main ideas for the proof below have been articulated in \cite{CS} for the number field case.  Here only a minor adjustment is required.  We start with reviewing two technical lemmas which can be found in \cite{Po}.  Let $E$ be an elliptic curve defined over a number field $K$ and fix an affine Weierstrass equation for the curve.  Let $P \in E(K)$ be a point of infinite order, let $n \in \Z_{\not=0}$, and let $(x_n, y_n)$ be the coordinates corresponding to $[n]P$ under the chosen Weierstrass model.  Given $x \in K$, let $\nn(x)$ be the integral divisor which is the numerator of the divisor of $x$ in $K$.  Further let $\dd(x) = \nn(x^{-1})$.

 \begin{lemma}
\label{le:anydivisor}%
Let $\mathfrak A$ be any integral divisor of $K$ and let $m$ be a positive integer.  Then there exists  $k \in \Z_{>0}$ such that $\mathfrak A \Big{|} \dd(x_{km})$ in the integral divisor semigroup of $K$.
\end{lemma}%

 \begin{lemma}%
\label{le:equiv}
There exists a positive integer $m$ such that for any positive integers $k,l $,
\[%
\dd(x_{lm}) \Big{|} \nn\left (\frac{x_{lm}}{x_{klm}}-k^2\right)^2
\]%
in the integral divisor semigroup of $K$. 
\end{lemma}%

The following proposition can be found in \cite{CS}.  We include its short proof for the convenience of the reader.
 \begin{proposition}%
\label{extensions}
Let $N/K$ be a number field extension of degree $n$.  Let $\mathfrak Q$ be a prime of $K$ and let
$\qq_1, \ldots, \qq_m$ be all the  primes of $N$ lying above $\mathfrak Q$.   Let $u
\in N$ be integral at $\mathfrak Q$.  Assume further  there exists a sequence $\{(k_i, y_i)\}$ where $k_i \in \Z_{>0}$,
$k_{i+1} >k_i$, $y_i \in K$ with $\ord_{\qq_j}y_i \geq 0$ for all $i$ and $j$, and such that for all $i,j$ we
have that $\ord_{\qq_j}(u-y_i) > k_i$.    Then $u \in K$.
\end{proposition}  %
\begin{proof}%
 Let $\alpha \in N$ be a
generator of $N$ over $K$ such that  $\alpha$ is  integral with respect to $\mathfrak Q$.  Let $D$ be the discriminant of the power basis of $\alpha$.
Using this power basis we can represent any $w \in N$ in the following form:
\[%
w=\sum_{r=0}^{n-1}b_r\alpha^r
\]%
with $Db_r \in K$ and integral at $\mathfrak Q$. Note that for some $a_0, a_1, \ldots, a_{n-1} \in K$ we have that 
\[%
u-y_i=(a_0-y_i) +\sum_{r=1}^{n-1}a_r\alpha^r
\]%
and
\[%
\ord_{\qq_j}(u-y_i) >k_i, j=1,\ldots, m.
\]%
Let $\ell$ be a positive integer and choose $i$ such that $k_i >n(\ell+\ord_{\mathfrak Q}D)$.  In this case 
\[
u-y_i \equiv 0 \mod \mathfrak Q^{\ell+\ord_{\mathfrak Q}D}
\]
in the integral closure of the valuation ring of $\mathfrak Q$ in $N$.  Let $B \in K$ be such that 
\[
\ord_{\mathfrak Q}B =\ell +\ord_{\mathfrak Q}D.
\]
  Observe that $\displaystyle \frac{u-y_i}{B}$ is integral at $\mathfrak Q$,  and therefore $D\frac{a_r}{B}$ is integral at $\mathfrak Q$ implying that $\ord_{\mathfrak Q}a_r \geq \ell$ for $r=1,\dots, n-1$. Since $\ell$ can be arbitrarily large, $a_r=0, r=1,\ldots,n-1$ and $u \in K$.
\end{proof}%
We now use  our results on defining integrality at a single number field prime to obtain the following proposition.
 \begin{theorem}%
Let $\pp_G$ be a  completely $q$-bounded  in $K_{\inff}$ prime of $G$.   If there exists an elliptic curve $E$ defined over $G$ such that $\rank (E(K_{\inff})) >0$ and $E(K_{\inff})=E(G)$, then
$G$ is first-order definable over $K_{\inff}$ with only one variable in the range of the universal quantifier.
\end{theorem}%
\begin{proof}%
Fix an affine Weierstrass equation $y^2=x^3+ax+c$ for $E$ and identify non-zero points of $E(K_{\inff})$ with pairs of solutions to the Weierstrass equations as above.  Let $b \in K_{\inff}$ be such that it satisfies conditions of Theorem \ref{finite}, Part 1 with respect to all prime factors of $\pp_G$ in $M(b)$, i.e. $\ord_{\pp_{M(b)}}b<0$ and $\ord_{\pp_{M(b)}}b\not \equiv 0 \mod q$ for all $\pp_{M(b)} \in \calC_{M(b)}(\pp_G)$, where $M$ is a completely bounding field for $\pp_G$.  Let $u \in K_{\inff}$ be such that $ub \in Int(b,\pp_G, q)$ and 
\[%
\forall z \in K_{\inff} \exists (a_1, b_1), (a_2, b_2)  \in E(K_{\inff}):
\]%
\begin{equation}
\label{elliptic}
\frac{b^2}{za_1} \in Int(b,\pp_G, q) \land (u-\frac{a_1}{a_2})^2a_1 \in Int(b,\pp_G, q).
\end{equation}
 We claim that if  the formula is true for some $u \in N=M(b,u)$, then,    by
Proposition \ref{extensions}, we have that $u \in G$.  Indeed, given a $z \in N$ and $\displaystyle \frac{b^2}{za_2} \in Int(b,\pp_G, q)$, we have that for all $\pp_{N}$ lying above $\pp_{G}$, it is the case that 
\[
\ord_{\pp_N}\frac{b^2}{za_1} > \left( \frac{q-1}{q}\right )\ord_{\pp_N}b
\]
implying
\[
-\ord_{\pp_N}z +\ord_{\pp_N}\frac{1}{a_1} > \left (\frac{q-1}{q}-2\right )\ord_{\pp_N}b =\left (-1-\frac{1}{q}\right )\ord_{\pp_N}b > -\ord_{\pp_N}b >0. 
\]
 So that we have 
 \[
 \ord_{\pp_N}\frac{1}{a_1} > \ord_{\pp_N}z -\ord_{\pp_N}b > \ord_{\pp_N}z.
 \]
 The second part of the conjunction in \eqref{elliptic} now implies
 \[
 \ord_{\pp_N}(u-\frac{a_1}{a_2})^2a_1 >\frac{q-1}{q}\ord_{\pp_N}b,
 \]
\[
2 \ord_{\pp_N}(u-\frac{a_1}{a_2})>\frac{q-1}{q}\ord_{\pp_N}b + \ord_{\pp_N}\frac{1}{a_1} >\frac{q-1}{q}\ord_{\pp_N}b -\ord_{\pp_N}b +\ord_{\pp_N}z>\ord_{\pp_N}z .
 \]
 Since $z$ can be any element of $N$ and $\frac{a_1}{a_2} \in G$, it follows at once from Proposition \ref{extensions} that $u \in G$.

Now assume that $u=k^2$ with $ k\in\Z$.    Let $(x_1,y_1) \in E(G)$ be the
affine coordinates with respect to a chosen Weierstrass equation  of a point $P \in E(G)$  of infinite order, as above.   Then
by Lemma \ref{le:equiv} there exists a positive integer $m$ such that for any positive integer
$l $,
\[%
\dd(x_{lm}) \Big{|} \nn\left (\frac{x_{lm}}{x_{klm}}-k^2\right)^2
\]%
in the integral divisor semigroup of $G$.  Further, by Lemma \ref{le:anydivisor} we have that for any positive $C$, for some  $r$ it is the case
that $\ord_{\pp_N}x_{rm} < -C$ for any $\pp_N$.  So given a $z \in K_{\inff}$, let $a_1=x_{rm}, a_2=x_{krm}$ with $r$ chosen so that $\dd(b^{2})\nn(z) | \dd(x_{rm})$ in the integral divisor semigroup of $G(b,z)$ and observe that the first part of the Conjunction \eqref{elliptic} is satisfied.  Next we note that for $N=G(b,z)$, since $\ord_{\pp_N}b <0$ we have that $\ord_{\pp_N}x_{rm}<0$, and since 
\[
\dd(x_{rm}) \Big{|} \nn\left (\frac{x_{rm}}{x_{krm}}-k^2\right)^2,
\]
 we also must have that 
\[
\ord_{\pp_N}\left (\left (\frac{x_{rm}}{x_{krm}}-k^2\right )^2x_{rm}\right ) \geq 0
\]
 and thus the second part of the conjunction \eqref{elliptic} is satisfied.

Finally we note that any positive integer can be
written as a sum of four squares, and any element of $G$ can be expressed as a linear combination of
some basis elements with rational coefficients.  The resulting formula for $G$ is of the form $\exists \ldots \exists \forall \exists \ldots \exists P$, where $P$ is a polynomial equation.
\end{proof}%
In view of the above proposition we have the following theorem.
\begin{theorem} 
\label{thm elliptic}
Let $q$ be a rational prime and let $K_{\inff}$ be an infinite algebraic extension of $\Q$ with at least one prime of a number field contained in $\Q$ completely $q$-bounded.  Assume also there exist an elliptic curve defined over $K_{\inff}$ such that its Mordell-Weil group has positive rank and is finitely generated.  In this case $\Z$ is first-order definable over this field and there fore the first-order theory of this field is undecidable.
\end{theorem}
This theorem provides another way to improve results due to C.7 Videla in \cite{V3}, where finitely generated elliptic curves are used over cyclotomics with one ramified rational prime to generate a model of $\Z$ using results of Julia Robinson.  Using these elliptic curves as described above we would also get the first-order definition of $\Z$ as a subset.  

    Another example of a family of infinite extensions of $\Q$ where one can find finitely generated elliptic curves can be found in \cite{Sh37} where the curves are used to prove existential undecidability of rings of integers.  One should note that the fields described in that paper are all $q$-bounded with respect to almost all rational primes and thus one could also derive the results on the first-order undecidability of these fields using the norm equation method above.  In general the full strength of the elliptic curves method is unknown since we don't have the complete picture concerning elliptic curves in infinite algebraic extensions of $\Q$.  

One can also use Theorem \ref{thm elliptic} to obtain information about existence of finitely generated curves in infinite extensions.  If an infinite extension of $\Q$ with a completely $q$-bounded prime has a decidable first-order theory, then our theorem implies that any elliptic curve defined over the field either has rank 0 or is not finitely generated.    An example of such a field, pointed out to us by Moshe Jarden, can be found in \cite{Er5}.  Fix a prime number $p$ and consider the field of all algebraic numbers $\Q^{alg}_p$ contained in $\Q_p$, the $p$-adic completion of $\Q$.  The field $\Q^{alg}_p$ is not fixed under conjugation and we can set $K_{\inff}$ to be the intersection of all the conjugates of $\Q^{alg}_p$ over $\Q$.  Ershov showed that the first-order theory of such a field is decidable.  Further, $p$ splits completely in every finite extension contained in such a $K_{\inff}$ and therefore it is $q$-bounded for any $q$.

\end{document}